\documentclass[12pt]{amsart}
\usepackage{fullpage}
\usepackage{tabu}
\usepackage{amssymb} 
\usepackage{amsmath} 
\usepackage{amscd}
\usepackage{amsbsy}
\usepackage{comment, enumerate}
\usepackage[matrix,arrow]{xy}
\usepackage{mathrsfs}
\usepackage{color}
\usepackage{mathtools,caption}
\usepackage{todonotes}
\usepackage{tikz-cd}
\usepackage[normalem]{ulem}

\definecolor{darkgreen}{rgb}{0,0.5,0}

\usepackage[
        colorlinks, citecolor=darkgreen,
        backref,
]{hyperref}

\newcommand{\Aff}{\mathbb{A}}
\newcommand{\C}{\mathbb{C}}
\newcommand{\F}{\mathbb{F}}

\newcommand{\Q}{\mathbb{Q}}

\newcommand{\Z}{\mathbb{Z}}
\newcommand{\OK}{\mathcal{O}_K}
\newcommand{\Qbar}{{\overline{\Q}}}

\newcommand{\rhobar}{{\bar{\rho}}}

\newcommand{\p}{\mathfrak{p}}
\newcommand{\q}{\mathfrak{q}}
\newcommand{\m}{\mathfrak{m}}
\newcommand{\Fl}{\mathfrak{l}}

\newcommand{\OL}{\mathcal O_L}
\newcommand{\rf}{\sqrt{5}}

\newcommand{\brhoJp}{\bar\rho_{J,\p}}
\newcommand{\brhoJmp}{\bar\rho_{J^-,\p}}
\newcommand{\brhoJmFr}{\bar\rho_{J^-,\Fr}}
\newcommand{\rhoJmFr}{\rho_{J^-,\Fr}}
\newcommand{\rhoJmp}{\rho_{J^-,\p}}
\newcommand{\rhoJpp}{\rho_{J^+,\p}}
\newcommand{\rhoJpmp}{\rho_{J^\pm,\p}}

\newcommand{\rhoJmptwo}{\rho_{J^-,\p} \otimes \chi_0}

\newcommand{\brhoJmptwo}{\rhobar_{J^-,\p} \otimes \chi_0}
\newcommand{\brhoJpp}{\bar\rho_{J^+,\p}}
\newcommand{\brhoJppsemi}{\bar\rho^{\text{ss}}_{J^+,\p}}

\newcommand{\brhoJpmp}{\bar\rho_{J^{\pm},\p}}

\newcommand{\overbar}[1]{\mkern 1.5mu\overline{\mkern-1.5mu#1\mkern-1.5mu}\mkern 1.5mu}

\newcommand{\calC}{\mathcal{C}}

\newcommand{\calO}{\mathcal{O}}

\newcommand{\Fp}{\mathfrak{p}}
\newcommand{\Fq}{\mathfrak{q}}
\newcommand{\Fr}{\mathfrak{r}}

\newcommand{\tilX}{\tilde{X}}

\DeclareMathOperator{\End}{End}

\DeclareMathOperator{\Frob}{Frob}
\DeclareMathOperator{\Gal}{Gal}

\DeclareMathOperator{\Spec}{Spec}

\DeclareMathOperator{\tr}{tr}

\newcommand{\GL}{\operatorname{GL}}

\newcommand{\Magma}{{\tt Magma}}

\numberwithin{equation}{section}

\newtheorem{theorem}{Theorem}[section]
\newtheorem{lemma}[theorem]{Lemma}
\newtheorem{corollary}[theorem]{Corollary}
\newtheorem{proposition}[theorem]{Proposition}

\theoremstyle{definition}
\newtheorem{definition}[theorem]{Definition}

\theoremstyle{remark}
\newtheorem{remark}[theorem]{Remark}

\begin{document}

\title[A modular approach to Fermat equations of signature $(p,p,5)$]{A modular approach to Fermat equations of signature $(p,p,5)$ using Frey hyperelliptic curves}

\date{\today}


\author{Imin Chen}

\address{Department of Mathematics, Simon Fraser University\\
Burnaby, BC V5A 1S6, Canada.} 
\email{ichen@sfu.ca}

\author{Angelos Koutsianas}

\address{Department of Mathematics, Aristotle University of Thessaloniki\\
54124, Thessaloniki, Greece.} 
\email{akoutsianas@math.auth.gr}

\begin{abstract}
In this paper we carry out the steps of Darmon's program for the generalized Fermat equation
$$
x^n + y^n = z^5.
$$
In particular, we develop the machinery necessary to prove an optimal bound on the exponent $n$ for solutions satisfying certain $2$-adic and $5$-adic conditions which are natural from the point of view of the method. We also reduce the problem of resolving this equation to a `big image conjecture', completing a line of ideas suggested in his original program.

The above equation is an example of a generalized Fermat equation for which the predicted Frey abelian varieties have dimension $ > 1$ and thus it represents an interesting test case for Darmon's program. 
\end{abstract}

\thanks{I.\ Chen was supported by NSERC Discovery Grant RGPIN-2017-03892.}

\maketitle

{
\hypersetup{linkcolor=black}
\setcounter{tocdepth}{1}
\tableofcontents
}

\section{Introduction}

This paper is motivated by the study of the generalized Fermat equation
\begin{equation}\label{general-equ}
   x^r + y^q = z^p.
\end{equation}
We say that a solution $(a,b,c) \in \Z^3$ to \eqref{general-equ} is \textit{primitive} if $\text{gcd}(a,b,c) = 1$, and \textit{trivial} if $abc = 0$. It is an open conjecture that \eqref{general-equ} has no non-trivial primitive solutions if $r,q,p \ge 3$ (see for instance, \cite{Bennett-open}).

The special case of $r = q = p$ is Fermat's Last Theorem which was proven in \cite{Wiles, Taylor-Wiles} using Galois representations and modularity. In \cite{DarmonDuke}, Darmon described a program to show the generalized Fermat equation \eqref{general-equ} has no non-trivial solutions for one varying prime exponent $p \ge 3$ using the approach of Galois representations and modularity.

Central to Darmon's program is the construction of Frey representations of signature $(r,q,p)$, which is done explicitly in \cite{DarmonDuke} for the case $(p,p,r)$ using Frey abelian varieties which are constructed as the Jacobians of Frey hyperelliptic curves. In order to carry out Darmon's program for a specific $r,q$, one needs to prove irreducibility and modularity of the $2$-dimensional residual Galois representations attached to a putative solution, as well as distinguish them from those of the trivial solutions. Due to recent breakthroughs in modularity results, establishing modularity can be carried out in many cases, but we note for the main results of this paper, we do in fact need the multi-Frey method in order to avoid a case where a suitable modularity result is unknown.

Although irreducibility cannot be established in all cases, using local methods, it is possible to treat certain congruence classes. In addition, the trivial solution $(\pm 1,\mp 1,0)$ presents an essential obstruction to the method because its associated Frey hyperelliptic curve is non-singular and the Jacobian of this hyperelliptic curve is modular at the Serre level. 

Unlike signature $(r,r,p)$ where one still has Frey elliptic curves due to \cite{F}, signature $(p,p,r)$ for $r \ge 5$ necessitates the consideration of Frey abelian varieties of dimension greater than $1$ if the exponents are prime \cite[\S 4.3]{Darmon-Granville}. For $r = 2, 3$, signature $(p,p,r)$ equations were resolved in \cite{Darmon-Merel} using Frey elliptic curves which exist in these cases. The equation studied in this paper thus represents an interesting test case for Darmon's program and develops some of the tools which in future studies of the case $r > 5$.

In this paper, we consider the specific signature $(p,p,5)$ and develop the necessary machinery to carry out Darmon's program in all congruence classes mod $10$ which avoid the two obstructions above. The method is sufficiently refined to establish optimal bounds on the exponents $p$.

\begin{theorem}\label{main-theorem}
For $n \ge 3$, there are no non-trivial primitive solutions $(a,b,c) \in \Z^3$ to the equation 
\begin{equation} \label{main-equ}
  x^n + y^n = z^5,
\end{equation}
in the cases
\begin{enumerate}
    \item[{\rm (I)}] $2 \nmid ab$ and $5 \mid ab$, or
    \item[{\rm (II)}] $2 \mid ab$ and $5 \nmid ab$.

\end{enumerate} 
\end{theorem}

We remark that the proof of the above theorem requires the use of both Frey abelian varieties $J_5^\pm(a,b,c)$ introduced in \cite{DarmonDuke} for signature $(p,p,5)$ in order to deal with a lack of a suitable modularity result for $J_5^+(a,b,c)$ in the case $5 \nmid ab$. 

As a consequence of the tools developed, another theorem we prove distills the obstructions to resolving \eqref{main-equ}.

\begin{theorem}\label{thm:eliminate-to-CM}
For prime $p$ sufficiently large, any non-trivial primitive solution $(a,b,c) \in \Z^3$ to the equation~\eqref{main-equ} gives rise to a residual Frey representation $\rhobar_{J_5^-(a,b,c),\Fp}$ such that $\rhobar_{J_5^-(a,b,c),\Fp}$ is reducible, or is a quadratic twist of $\rhobar_{J_5^-(\alpha, -\alpha, 0),\Fp}$ for some non-zero integer $\alpha$, where $\Fp$ is any choice of prime of $K = \Q(\zeta_5)^+$ above $p$.
\end{theorem}
Since any quadratic twist of $J_5(\alpha, -\alpha, 0)$ has CM by $\Q(\zeta_5)$, the residual representation has small image, so Darmon's conjectural ideas can now kick in.

\begin{corollary}\label{assume-conj}
Assuming Conjecture 4.1 in \cite{DarmonDuke}, there are no non-trivial primitive solutions to
\begin{equation*}
    x^p + y^p = z^5,
\end{equation*}
for prime $p$ sufficiently large.
\end{corollary}

The above corollary was not proven in \cite{DarmonDuke} because modularity of $J_5^+(a,b,c)/K$ cannot be achieved for $5 \nmid ab$, while our study of $J_5^-(a,b,c)/K$ allows us to prove its modularity without any conditions. We also remark that we only need Conjecture 4.1 in the Borel case when $10 \mid ab$ as irreducibility in the other cases were established in the course of proving Theorem~\ref{main-theorem}.

The obstruction to completely resolving \eqref{main-equ} thus consists of proving irreducibility in the case $2 \mid ab, 5 \mid ab$ and eliminating the CM-form in the case $2 \nmid ab, 5 \nmid ab$. These are two fundamental obstructions in the study of generalized Fermat equations of signature $(p,p,r)$ that still remain.

The innovations needed to achieve the main results are:
\begin{enumerate}
    \item We provide efficient and general local methods which can be used for establishing irreducibility of residual Galois representations attached to Frey abelian varieties.
    
    \item The conductor at primes above $2$ of a Frey abelian variety is generally more difficult to determine. For instance, the methods in \cite{DokchitserDokchitserMaistretMorgan19, DDMM-local, maistret} do not apply. We show that it can be read off easily up to twist by using the relation between odd Frey abelian varieties and the Legendre family of elliptic curves, thereby ``propagating the computation of conductors'' from the Legendre family. This differs from the approach in \cite{BillereyChenDieulefaitFreitasNajman} where explicit computations with hyperelliptic models are performed to determine the inertial type.
    
    \item We apply repeated elimination steps using several twists of the Frey abelian variety to minimize the conductor at primes above $2$ and avoid the need for inertia arguments such as in \cite{Chen-2022-x131337}.
    
    \item We give a precise conductor calculation at the prime above $5$ of the Frey abelian variety $J_5^-(a,b,c)$ by identifying extensions which achieve semistable reduction, giving first examples of such computations in the ``wild case'' for parameterized families of Frey abelian varieties of dimension $> 1$.
\end{enumerate}

The programs and output transcripts for the computations needed in this paper are described and posted at \cite{programs}.

\section{Acknowledgements}

We thank Nicolas Billerey, Nuno Freitas, Luis Dieulefait, and C\'eline Maistret for helpful discussions and suggestions pertaining to the subject of this paper. We also thank Mohammad Sadek for discussions about \cite[Th\'eor\`eme 1 (I)]{Liu93}. The second author is grateful to the Department of Mathematics of The University of British Columbia, especially to Michael Bennett, because part of this work was accomplished during his time in Vancouver.

\section{Hyperelliptic equations}

In this section, we summarize the basic theory of hyperelliptic equation, which is taken in part from \cite{Liu93, Liu96, Liu-book, Lockhart}. In this paper, a curve $C$ over $K$ is an integral scheme of dimension $1$ which is proper over $K$.

Let $K$ be a finite extension of $\Q_\ell$ where $\ell$ is a rational prime. Denote by $\calO$ the ring of integers of $K$, by $k$ the residue field of $\calO$ and by $v$ the valuation of $K$. A \textit{hyperelliptic equation $E$ over $K$} is an equation of the form
\begin{equation}
  y^2 + Q(x) y = P(x),
\end{equation}
where $Q, P \in K[x]$, $\deg Q \le g+1$, and $\deg P \le n = 2g + 2$ with
\begin{equation}
    2 g + 1 \le \max(2 \deg Q, \deg P) \le 2 g + 2.
\end{equation}
Let
\begin{equation*}
  R = 4 P + Q^2,
\end{equation*}
and suppose $\kappa$ is the leading coefficient of $R$. The \textit{discriminant} of $E$ \cite{Liu96} is given by
\begin{equation*}
   \Delta_E := \begin{cases}
    2^{-4(g+1)} \Delta(R) & \text{ if } \deg R = 2 g + 2, \\
    2^{-4(g+1)} \kappa^2 \Delta(R) & \text{ if } \deg R = 2 g + 1,
\end{cases}
\end{equation*}
where $\Delta(H)$ denotes the discriminant of $H \in K[x]$. In particular, if $P$ is monic, $\deg P = 2 g + 1$, and $\deg Q \le g$, then
\begin{equation*}
  \Delta_E = 2^{4g} \Delta(P + Q^2/4),
\end{equation*}
using the fact that $\Delta(H)$ is homogeneous of degree $2n - 2$ in the coefficients of $H$.

 A hyperelliptic equation $E$ over $K$ gives rise to a curve $C$ over $K$ by gluing the following two affine schemes over $K$
\begin{align}
\label{glue-model}
   & \Spec K[x,y]/(y^2 + Q(x) y - P(x)) \\
   & \Spec K[u,z]/(z^2 + T(u) z - S(u))
\end{align}
along the open sets $x \not= 0$ and $u \not= 0$, respectively, by the relations
\begin{equation}
  x = 1/u, \quad y = z/u^{g+1},
\end{equation}
where
\begin{equation}
\label{relation-model}
    P(x) = u^{2g+2} S(u), \quad Q(x) = u^{g+1} T(u).
\end{equation}

\begin{definition}
A curve $C$ over $K$ given by a hyperelliptic equation $E$ over $K$ such that $\Delta_E \not= 0$ is called \textit{a hyperelliptic curve over $K$} \footnote{An alternative definition is that $C$ is a smooth geometrically connected curve over $K$ with a finite separable morphism of degree $2$ to $\mathbb{P}^1_K$ (see \cite[Chapter 7, Definition 4.7]{Liu-book})}. A hyperelliptic equation $F$ with coefficients in $\calO$ which gives rise to a curve over $K$ which is isomorphic to $C$ over $K$ is a called \textit{a hyperelliptic model over $\calO$} for $C$ over $K$. 
\end{definition}
The quantity $g$ is the genus of the curve $C$ over $K$. Two hyperelliptic equations $E, F$ for a hyperelliptic curve $C$ over $K$ are related by the transformations
\begin{align}
\label{hyperelliptic-transform}
   & E : y^2 + Q(x) y = P(x), \notag \\
   & F: z^2 + T(u) z = S(u), \notag \\
   & \begin{pmatrix} a & b \\ c & d \end{pmatrix} \in \GL_2(K), \quad e \in K^*, \quad H(u) \in K[u], \quad \deg(H) \le g+1, \notag \\
   & x = \frac{a u +b}{c u + d}, \quad y = \frac{e z + H(u)}{(c u + d)^{g+1}}, \\
   & \Delta_F = \Delta_E \, e^{-4(n-1)} (ad-bc)^{n(n-1)}. \notag
\end{align}
In particular, we note that the valuation of the discriminant modulo
\begin{equation}
  \gcd(4(n-1),n(n-1)),
\end{equation}
is an invariant of the isomorphism class of $C$ over $K$.

We say a \textit{hyperelliptic equation has odd degree} if $P(x)$ is monic and
\begin{equation*}
     \deg P = 2 g + 1, \quad \deg Q \le g,
\end{equation*}
An odd degree hyperelliptic curve over $K$ or model over $\calO$ is one which is given by an odd degree hyperelliptic equation.

Two odd degree hyperelliptic equations~$E : y^2 + Q(x) y = P(x)$ and~$F : z^2 + T(u) z = S(u)$ for the same hyperelliptic 
curve~$C$ over~$K$ are related by a transformation of the shape
\begin{align*}
   & x = e^2 u + r, \quad y = e^{2g+1} z + t(u), \quad \text{where} \\
   & e \in K^*,\quad  r \in K,\quad t \in K[u],\quad \deg(t) \le g. 
\end{align*}
The discriminants of the odd degree hyperelliptic equations are related by
\begin{equation}\label{odd-discrim-trans}
\Delta_F = \Delta_E \, e^{4g(2g+1)},
\end{equation}
hence the valuation of the discriminant modulo $4g(2g+1)$ is an invariant of the isomorphism class of the pointed hyperelliptic curve $(C,P)$ over $K$, where $P = \infty$ is the given $K$-rational Weierstrass point on $C$ as it is described by an odd degree hyperelliptic equation (see \cite{Lockhart} for more details). 

\begin{definition}\label{def:model}
A \textit{model} $\calC$ over $\calO$ for a curve $C$ over $K$ is a $\calO$-scheme which is proper and flat over $\calO$ such that $\calC_K \cong C$ where $\calC_K$ is the generic fiber of $\calC$.
\end{definition}
We note that a hyperelliptic model over $\calO$ for $C$ over $K$ gives rise to a model over $\calO$ for $C$ over $K$ as in \eqref{glue-model}-\eqref{relation-model}.

\begin{definition}\label{def:model_good_mult}
A model $\calC$ over $\calO$ for a curve $C$ over $K$ has \textit{good reduction} if and only if its reduction mod $v$ is non-singular over $k$. In addition, we say that $\calC$ has \textit{bad semistable reduction} if and only if its reduction mod $v$ is reduced, singular, and has only ordinary double points as singularities.
\end{definition}

\begin{definition}\label{def:curve_good_mult}
A curve $C$ over $K$ has \textit{good reduction} (resp.\ \textit{bad semistable reduction}) if and only if there is some model $\calC$ over $\calO$ for $C$ which has good reduction (resp.\ bad semistable reduction). We say that $C$ has \textit{semistable reduction} if and only if it has good or bad semistable reduction.
\end{definition}

\begin{definition}
Suppose $\calC:~y^2 + Q(x)y = P(x)$ is a hyperelliptic model over $\calO$ of a hyperelliptic curve over $K$. If $Q(x)=b_{g+1}x^{g+1} + \cdots + b_0$ and $P(x)=a_nx^n + \cdots + a_0$, then we define the \textit{valuation vectors over} $K$ of this hyperelliptic model as the pair of vectors
\begin{equation*}
(v(a_0),\cdots,v(a_n))\qquad (v(b_0),\cdots,v(b_{g+1})).
\end{equation*}
\end{definition}

\begin{proposition}\label{hyperelliptic-good}
Let $C$ be a hyperelliptic curve over $K$ with a $K$-rational point $P$. Then $C$ has good reduction if and only if $C$ has an odd degree hyperelliptic model $\calC$ over $\calO$ such that $v(\Delta(\calC)) = 0$.
\end{proposition}
\begin{proof}
If $C$ has an odd degree hyperelliptic model $\calC$ over $\calO$ such that $v(\Delta(\calC)) = 0$, then $\calC$ is a model with good reduction so $C$ has good reduction.

Suppose $C$ has good reduction, so there exists a model $\calC$ of $C$ with good reduction. By \cite[Exercise 8.3.6]{Liu-book}, the hyperelliptic map
\begin{equation*}
  \pi: C \rightarrow \mathbb{P}^1_K,
\end{equation*}
extends to
\begin{equation*}
  \pi : \calC \rightarrow \mathbb{P}^1_\calO.
\end{equation*}

As a result, $\calC_k$ is a non-singular pointed hyperelliptic curve with a $k$-rational Weierstrass point $\tilde{P}$ where $\tilde{P}$ is the reduction of $P$ mod $v$, where $P$ is the $K$-Weierstrass point on $C$ over $K$, so using \cite[Proposition 1.2]{Lockhart} it follows that $\calC_k$ can be given by a non-singular odd degree hyperelliptic equation. We deduce that $C$ has an odd degree hyperelliptic model $\calC$ such that $v(\Delta(\calC))=0$.
\end{proof}
More generally, see \cite[Chapter 7, Remark 4.25]{Liu-book}.

\section{Abelian varieties}\label{sec:abelian_varieties}

In this section, we recall definitions and results about $\GL_2$-type abelian varieties and semistable reduction of abelian varieties that we use later in the paper. We recommend \cite[Section II]{RibetGalois} for an introduction to $\GL_2$-type abelian varieties.

Let $K$ be a number field and $G_K = \Gal(\overbar{K}/K)$ denotes the absolute Galois group of $K$, where $\overbar{K}$ is an algebraic closure of $K$. Let $A$ be an abelian variety over $K$ of dimension $g$. 

Let $E$ be a number field of degree $g$ over $\Q$. We say that \textit{A is of $\GL_2(E)$-type over $K$} if 
\begin{equation*}
    E\hookrightarrow(\End_{K}A)\otimes\Q,
\end{equation*}
where $\End_{K}(A)$ is the ring of endomorphisms of $A$ defined over $K$. 

Let $\ell$ be a rational prime and $T_\ell(A)$ the Tate module of $A$ with $V_{\ell}(A)=T_\ell(A)\otimes\Q_\ell$. As $G_K$-modules we have that
\begin{equation}\label{tate-decomp}
    V_\ell(A) \cong \oplus_{\lambda \mid \ell} V_\lambda(A),
\end{equation}
where $V_\lambda(A) = V_\ell(A) \otimes_{E_\ell} E_\lambda$ is the $\lambda$-adic Tate module of $A$, $E_\ell = \oplus_{\lambda \mid \ell} E_\lambda$, and $E_\lambda$ is the completion of $E$ at $\lambda$. Because $G_K$ acts $E_\lambda$-linearly on $V_{\lambda}(A)$ and $V_{\lambda}(A)$ is a $E_\lambda$-vector space of dimension $2$ we get a $\lambda$-adic Galois representation
\begin{equation}
    \rho_{A,\lambda}:~G_K\rightarrow\GL_2(E_{\lambda}).
\end{equation}

For the system $(\rho_{A,\lambda})$ of Galois representations of $A$ we have the following theorem \cite[Section 11.10]{Shimura67}.

\begin{theorem}\label{thm:weakly_compatible_system}
The system $(\rho_{A,\lambda})$ of Galois representations of $A$ is a \textit{$E$-rational weakly compatible system\footnote{See \cite{Boeckle} for the definition of $E$-rational weakly (strictly) compatible system of Galois representations.}}. In particular, the Frobenius polynomials associated with $(\rho_{A,\lambda})$ all have coefficients in the ring of integers of $E$.
\end{theorem}

Let $w$ be a prime in $K$ over $A$ has good reduction. We denote by $a_w$ the \textit{trace of Frobenius} of $w$ and the system $(\rho_{A,\lambda})$. We define the \textit{Frobenius field $K_A$ of $A$} to be the field generated by the traces $a_w$. From Theorem \ref{thm:weakly_compatible_system}, $K_A$ is a subfield of $E$.

At the end of this section, we recall the definition of semistable and multiplicative reduction of an abelian variety and how we can get this characterization when $A$ is the Jacobian of a curve $C$. 

\begin{definition}
An abelian variety over $K$ has \textit{semistable reduction} if and only if the linear part of the special fiber of the connected component of its N\'eron model is an algebraic torus. Furthermore, if its toric rank is positive, we say it has \textit{multiplicative reduction}.
\end{definition}

\begin{theorem}\label{thm:C_J_semistable_reduction}
Let $C$ be a curve over $K$ and let $J$ be the Jacobian of $C$. Then $C/K$ has semistable reduction if and only if $J/K$ has semistable reduction. Furthermore, if $C/K$ has bad semistable reduction with a model $\calC$ that has integral special fiber, then $J/K$ has multiplicative reduction.
\end{theorem}

\begin{proof}
See \cite[Theorem 2.4]{deligne-mumford} for the first assertion. Let $X$ be the special fibre of $\calC$. Then, by assumption $X$ is integral and has only ordinary double points as singularities. Therefore, with the notation in \cite[Section 3.3]{Romagny}, it holds that $X' = X_\text{red} = X$. If $\tilX$ is the normalization of $X$ then $\tilX$ is connected because $X$, and therefore $\tilX$, is irreducible. Hence, $c = 1$ in \cite[Lemma 3.3.5]{Romagny}, so the toric rank of the reduction of $J$ is positive. Thus, $J$ has multiplicative reduction.
\end{proof}


\section{Darmon's Frey hyperelliptic curve for signature $(p,p,r)$}

In this section, we briefly review from \cite{DarmonDuke,ttv} the construction of a suitable Frey hyperelliptic curve for the equation
\begin{equation}
\label{rrp-equ}
  x^p + y^p = z^r,
\end{equation}
where $r$ and $p$ are odd primes.

Let $\zeta_r$ be a primitive $r$th root of unity, $\omega_j = \zeta_r^j + \zeta_r^{-j}$, and put
\begin{equation*}
  g(x) = \prod_{j=1}^{\frac{r-1}{2}} (x + \omega_j).
\end{equation*}

From here on, we let $K = \Q(\zeta_r)^+$ and denote by $\mathfrak{r}$ the unique prime above $r$ in $K$. We also use the notation $\Fq_q$ to denote a prime of $K$ above $q$ and in particular, $\mathfrak{r} = \Fq_r$, where the former is used in the context of coefficients for representations, and the latter when completing $K$ at $\Fq_r$.

\begin{proposition}
The quotient of the hyperelliptic curve
\begin{equation*}
  Y^2 = X (X^{2r}+t X^r + 1),
\end{equation*}
by the involution $\tau : (X,Y) \mapsto (1/X,Y/X^{r+1})$ is given by the hyperelliptic curve
\begin{equation*}
  y^2 = x g(x^2-2) + t.
\end{equation*}
\end{proposition}
\begin{proof}
  See \cite[Proposition 3]{ttv}.
\end{proof}

\begin{proposition}
The quotient of the hyperelliptic curve
\begin{equation*}
  Y^2 = X^{2r}+t X^r + 1,
\end{equation*}
by the involution $\tau : (X,Y) \mapsto (1/X,Y/X^r)$ is given by the hyperelliptic curve
\begin{equation*}
  y^2 = (x+2)(x g(x^2-2) + t).
\end{equation*}
\end{proposition}
\begin{proof}
  See \cite[Remark, p.\ 1058]{ttv}.
\end{proof}

Consider the following hyperelliptic curves
\begin{align*}
  C_r^-(t) : & \quad y^2 = f_t^-(x) := f(x) + 2 - 4t, \\
  C_r^+(t) : & \quad y^2 = f_t^+(x) := (x+2)(f(x) + 2 - 4t),
\end{align*}
where $f(x) = xg(x^2-2)$.

\begin{theorem}
\label{discriminant-t}
  The discriminants of the polynomials $f_r^-(x)$ and $f_r^+(x)$ are given by
  \begin{align*}
    & \Delta(f^-_t) = (-1)^\frac{r-1}{2} 2^{2 (r-1)} r^r t^\frac{r-1}{2} (t-1)^\frac{r-1}{2}, \\
    & \Delta(f^+_t) = (-1)^\frac{r-1}{2} 2^{2 (r+1)} r^r t^\frac{r+3}{2} (t-1)^\frac{r-1}{2}.
  \end{align*}
\end{theorem}
\begin{proof}
  This is stated in slightly different form in the proof of \cite[Proposition 1.15]{DarmonDuke}. For further details on how such formulae can be justified, we refer the reader to \cite{BillereyChenDieulefaitFreitasNajman}.
\end{proof}

\begin{theorem}\label{thm:endomorphism_ring}
  Let $J_r^\pm(t)$ be the Jacobian of the hyperelliptic curve $C_r^\pm(t)$. Then, the endomorphism algebra $\End_K (J_r^\pm(t)) \otimes \Q$ contains the field $K$.
\end{theorem}

\begin{proof}
  For $J_r^-(t)$, see \cite[Theorem 1]{ttv}. For $J_r^+(t)$, the result follows from \cite[Remark, p.\ 1058]{ttv} and modifying the argument in \cite[\S 3.1]{ttv}.
\end{proof}

Let $(a,b,c) \in \Z^3$ be a non-trivial primitive solution to \eqref{rrp-equ}. The following lemma is readily verified and appears in \cite[p. 425]{DarmonDuke}.
\begin{lemma}
\label{lem:Frey-model}
Let
\begin{align*}
  C_r^-(a,b,c): & \quad y^2 = c^r f(x/c) - 2(a^p-b^p), \\
  C_r^+(a,b,c): & \quad y^2 = (x+2c) (c^r f(x/c) - 2 (a^p - b^p)).
\end{align*}
Then
\begin{enumerate}
\item $C_r^-(a,b,c)$ is isomorphic to a twist of $C_r^-(t)$ over $\Q$ where $t = a^p/c^r$,
\item $C_r^+(a,b,c)$ is isomorphic to $C_r^+(t)$ over $\Q$ where $t = a^p/c^r$.
\end{enumerate}
\end{lemma}

Some explicit examples are listed below.

$C_r^-(a,b,c)$ for:
\begin{align*}
r = 3: & \quad y^2 = x^3 - 3 c^2 x - 2 (a^p - b^p), \\
r = 5: & \quad y^2 = x^5 - 5 c^2 x^3 + 5 c^4 x - 2 (a^p - b^p), \\
r = 7: & \quad y^2 = x^7 - 7 c^2 x^5 + 14 c^4 x^3 - 7 c^6 x - 2 (a^p - b^p), \\
r = 11: & \quad y^2 = x^{11} - 11 c^2 x^9 + 44 c^4 x^7 - 77 c^6 x^5 + 55 c^8 x^3 - 11 c^{10} x - 2(a^p - b^p).
\end{align*}

$C_r^+(a,b,c)$ for:
\begin{align*}
r = 3: & \quad y^2 = (x + 2c) (x^3 - 3 c^2 x - 2 (a^p - b^p)), \\
r = 5: & \quad y^2 = (x + 2c) (x^5 - 5 c^2 x^3 + 5 c^4 x - 2 (a^p - b^p)), \\
r = 7: & \quad y^2 = (x + 2c) (x^7 - 7 c^2 x^5 + 14 c^4 x^3 - 7 c^6 x - 2 (a^p - b^p)), \\
r = 11: & \quad y^2 = (x + 2c) (x^{11} - 11 c^2 x^9 + 44 c^4 x^7 - 77 c^6 x^5 + 55 c^8 x^3 - 11 c^{10} x - 2(a^p - b^p)).
\end{align*}

\begin{remark}
The curve $C_3^+(a,b,c)$ corresponds to the (classical) Frey elliptic curve of the Fermat equation with signature $(p,p,3)$ which has been studied in \cite{Darmon-Merel, BenVatYaz}.
\end{remark}

From Theorem~\ref{discriminant-t} and Lemma~\ref{lem:Frey-model}, the discriminants of the hyperelliptic curves $C_r^-(a,b,c)$ and $C_r^+(a,b,c)$ are given by
\begin{align}
  \Delta(C_r^-) & = 2^{4(r-1)} r^r a^{p (r-1)/2} b^{p (r-1)/2}\label{disc-1} \\
  \Delta(C_r^+) & = 2^{4r} r^r a^{p (r+3)/2} b^{p (r-1)/2} \label{disc-2}.
\end{align}

We denote by $J_r^\pm = J_r^\pm(a,b,c)$ the Jacobians of $C_r^\pm(a,b,c)$. The abelian varieties $J_r^\pm/K$ are of $\GL_2(K)$-type \cite{DarmonDuke}.
\begin{remark}
  $C_r^\pm(1,-1,0)$ is non-singular and $J_r^\pm(1,-1,0)$ has complex multiplication by $\Q(\zeta_r)$ \cite[Proposition 3.7]{DarmonDuke}.
\end{remark}

\section{Modularity of $\rhoJpmp$}

From this section onward, we specialize to the case $r = 5$. Let $(a,b,c) \in \Z^3$ be a non-trivial primitive solution to \eqref{rrp-equ} for $r=5$, $K=\Q(\zeta_5)^+$ and $p$ is an odd prime. For convenience, we denote $C^\pm = C^\pm_5(a,b,c)$ and $J^\pm = J^\pm_5(a,b,c)$. 

Let $T_p(J^\pm)$ be the Tate module of $J^\pm$ and write $V_p(J^\pm) = T_p(J^\pm) \otimes \Q_p$. As $G_K$-modules from Section \ref{sec:abelian_varieties} and Theorem \ref{thm:endomorphism_ring} we have that
\begin{equation}
    V_p(J^\pm) \cong \oplus_{\Fp \mid p} V_\Fp(J^\pm),
\end{equation}
where $V_\Fp(J^\pm) = V_p(J^\pm) \otimes K_\Fp$ is the $\Fp$-adic Tate module of $J^\pm$, $K_p = \oplus_{\Fp \mid p} K_\Fp$, and $K_\Fp$ is the completion of $K$ at $\Fp$.

Let $\rhoJpmp$ be the Galois representation of $G_K$ on the $\p$-adic Tate module of $J^\pm$, where $\p$ is a prime of $K$ above $p$. We also denote by $\brhoJpmp$ the Galois representation of $G_K$ on the $\p$-torsion of $J^\pm$. 


Let $L$ be a totally real field and $n=[L:\Q]$. Let $f$ be a Hilbert newform over $L$ of weight $k=(k_1, \cdots, k_n)\in (\Z_{\geq 1})^n$, level $N_f$ an ideal of $\OL$ and character $\chi$. We denote by $K_f$ the eigenvalue field of $f$. Then for any prime $\Fl$ of $K_f$ above a rational prime $\ell$ there exists\footnote{The existence of $\rho_{f,\lambda}$ is due to the work of many people; \cite[Chapter 7]{Shimura71} and \cite{Eichler54,Shimura58, Deligne71, Deligne-Serre74} for $L=\Q$ and \cite{Ohta83, Rogawski-Tunnell83, Carayol86, Wiles88, Taylor89, BlasiusRogawski89}. The form of these Galois representations after restriction to a decomposition group is provided by \cite{Langlands73,Carayol86, Wiles88, Taylor89, Saito09}.} a continuous representation $\rho_{f,\Fl}: G_L \rightarrow \GL_2(K_{f,\Fl})$ which is unramified outside $N_f\ell$ and for any prime $\Fq\nmid N_f\ell$ it holds
\begin{align*}
    \tr \rho_{f,\Fl}(\Frob_{\Fq}) & = a_{\Fq}(f),\\
    \det \rho_{f,\Fl}(\Frob_{\Fq}) & = \chi(\Fq)N(\Fq)^{k_0-1},
\end{align*}
where $k_0=\max \left\{ k_1, \ldots, k_n \right\}$ and $K_{f,\Fl}$ is the completion of $K_f$ at $\Fl$ (for the description of $k_0$, see \cite{Shimura78}).

\begin{definition}
Let $A/L$ be an abelian variety of $\GL_2(E)$-type where $L$ is totally real fields and $E$ is a number field. We say $A/L$ is \textit{modular} if the weakly compatible system of representations $(\rho_{A,\lambda})$, as $\lambda$ varies over the finite primes of $E$, arises from the strictly compatible system of representations attached to a Hilbert newform $f$ over $L$. That is, the field of coefficients $K_f$ of $f$ is equal to the Frobenius field $L_A$ of $A$ and $\rho_{A,\lambda} \simeq \rho_{f,\Fl}$ for all primes $\Fl$ of $K_f$ and all primes $\lambda\mid\Fl$ of $E$.
\end{definition}

\begin{remark}\label{rem:modularity_for_all_primes}
    If $\rho_{A,\lambda} \simeq \rho_{f,\Fl}$ for some pair of primes ($\lambda$, $\Fl$), then in fact $\rho_{A,\lambda} \simeq \rho_{f,\Fl}$ for all pairs ($\lambda$, $\Fl$) by \cite[Theorem (1.3.1)]{RibetGalois}.
\end{remark}

Suppose $A/L$ is a modular abelian variety of $\GL_2(E)$-type. The conductor exponent at a prime $\Fq$ of $L$ of $\rho_{A,\lambda}$ is independent of $\lambda \nmid N(\Fq)$ which permits a definition of the conductor of a strictly compatible system of representations $(\rho_{A,\lambda})$. The conductor of the $\rho_{A,\lambda}$ coincides with the level of the Hilbert newform $f$ by \cite{Carayol86, Taylor89}.

\begin{lemma}\label{trivial-char}
The representations $\rhoJpp$ and $\rhoJmp$ have determinant the $p$-adic cyclotomic character $\chi_p$.
\end{lemma}
\begin{proof}
  This follows from \cite[Lemma 2.2.5, Corollary 2.2.10]{Wu-GL2}.
\end{proof}

\begin{theorem}\label{modularity-plus}
Suppose $5 \mid ab$. Then $J^+/K$ is modular.
\end{theorem}

\begin{proof}
This is \cite[Theorem 2.9]{DarmonDuke}.
\end{proof}

In the case $5 \nmid ab$, there is no currently known modularity result we can apply to prove that $\rhoJpp$ is modular in all cases. Therefore, we will work instead with $\rhoJmp$.

We denote by $\Fr$ the unique prime ideal of $K$ above $5$ when $K$ is regarded as the coefficient field for a representation.

\begin{proposition}\label{irreducible-minus-5} The representation $\brhoJmFr$ extends to $G_\Q$ and is absolutely irreducible when restricted to $G_{K(\zeta_5)}$.
\end{proposition} 

\begin{proof}
The abelian variety $J^-/K$ is of $\GL_2(K)$-type with $K\hookrightarrow \End_K(J^-)\otimes\Q$. Thus, the $5$-adic Tate module $T_5(J^-)\otimes\Q_5$ is a $2$-dimensional $K\otimes\Q_5$-vector space. As $5$ totally ramifies in $K$, we have that $K\otimes\Q_5 \simeq K_{\Fr}$.

Since $J^-$ is defined over $\Q$, as Darmon explains in \cite[p.\ 443]{DarmonDuke}, the action of $G_K$ on $T_5(J^-)\otimes\Q_5$ as a $K_{\Fr}$-vector space extends to a $G_\Q$-action that is $G_K$-semilinear; it satisfies
\begin{equation*}
    \sigma(\alpha v)=\sigma(\alpha) \sigma(v),
\end{equation*}
where $\sigma \in G_\Q$, $\alpha \in K_{\Fr}$ and $v\in T_5(J^-)$. As $\Fr$ is the unique prime above $5$, the action of $G_\Q$ on $T_5(J^-)\otimes_{\mathcal O_\Fr} \F_5$, where the tensor product is taken with respect to the reduction map $\mathcal O_\Fr\rightarrow \F_5$, is linear and restricts to the action of $G_K$ given by $\brhoJmFr$. 

By \cite[Theorem 2.6]{DarmonDuke}, we have that $\brhoJmFr$ arises up to quadratic twist from the Legendre family given by 
\begin{equation}
\label{L-curve}
  \mathcal{L} : y^2 = x (x-1) (x-t),
\end{equation}  
where $t = a^p/c^5$. More precisely, there is a character $\chi : G_K \rightarrow \F_{\Fr}^\times$ of order dividing $2$ such that 
\begin{equation}
\label{L-isom}
  \brhoJmFr\simeq \rhobar_{\mathcal{L},5} \otimes \chi.
\end{equation}
 
Suppose that $\brhoJmFr \mid_{G_{K(\zeta_5)}}$ is reducible; thus $\rhobar_{\mathcal{L},5}|_{G_{K(\zeta_5)}}$ is reducible. We then obtain a non-cuspidal $K(\zeta_5)$-rational point in $X_0(20)$. A short \Magma~program allows one to verify the $K(\zeta_5)$-points on $X_0(20)$ are all cuspidal, a contradiction.

We have thus shown that $\brhoJmFr \mid_{G_{K(\zeta_5)}}$ is irreducible. To show absolute irreducibility, we need to check that $\rhobar_{\mathcal{L},5} \mid_{G_{K(\zeta_5)}}$ cannot have image lying in a non-split Cartan subgroup $C'$, the only possible image for which the representation is irreducible, but becomes reducible after an extension of the coefficient field. The only possible image for $\rhobar_{\mathcal{L},5} \mid_{G_K}$ is the normalizer of $C'$. Thus, for $\mathcal{L}$ to have this property we must have that
\begin{equation}
    j_{\mathcal{L}}(t) - 1728 = j_{5N'}(s) - 1728.
\end{equation}
where
\begin{equation}
    j_{5N'}(s) = \frac{125 s (2s+1)^3 (2s^2+7s+8)^3}{(s^2+s-1)^5}
\end{equation}
is the $j$-invariant of elliptic curves with normalizer of non-split Cartan structure on $5$-torsion points \cite[Corollary 5.3]{Chen-siegel}. Since the left hand side is a square, $\mathcal{L}$ would give rise to a $K$-rational point on the hyperelliptic curve
\begin{equation}
    y^2 = 2 \left( x^2 + \frac{7}{2}x + \frac{27}{8} \right) (x^2+x-1).
\end{equation}
Using {\tt Magma}, it can be checked that all $K$-rational points arise from cusps.
\end{proof}

\begin{theorem}\label{modularity-minus}
The $J^-/K$ is modular.
\end{theorem}
\begin{proof}
By Proposition~\ref{irreducible-minus-5}, the representation $\brhoJmFr$ is absolutely irreducible and extends to a representation $\rhobar$ of $G_\Q$. By Serre's Conjecture over $\Q$, now proven in \cite{KhareWintenberger09a,KhareWintenberger09b,Kisin09}, $\rhobar$ is modular, hence $\brhoJmFr$ is also modular by cyclic base change. Modularity of $\rhoJmFr$ now follows from \cite[Theorem~1.1]{KhareThorne} by checking its three hypotheses:
\begin{enumerate}
\item The representation $\rhoJmFr$ is unramified outside the finite set of primes of $K$ dividing $10 \Delta(C^-)$.
\item The abelian variety $J^-/K$ is potentially semi-stable so $\rhoJmFr$ is deRham and hence Hodge-Tate. The Hodge-Tate weights of $\rho_{J^-,\Fr}$ are $\left\{ 0,1 \right\}$. 
\item The representation $\rhoJmFr\mid_{G_{K(\zeta_5)}}$ is absolutely irreducible from Proposition~\ref{irreducible-minus-5}.
\end{enumerate}

Since $\rhoJmFr$ is modular from Remark \ref{rem:modularity_for_all_primes} we get that $\rho_{J^-,\lambda}$ is modular for every finite prime $\lambda$ of $K$. Hence, $J^-$ is modular.
\end{proof}

\section{Conductors of $\brhoJpmp$}

The discriminants of the hyperelliptic curves $C^\pm/K$ are given by
\begin{align}\label{discriminant-5}
    \Delta(C^-)  & = 2^{16} 5^5 (ab)^{2p}, \\
    \Delta(C^+)  & = 2^{20} 5^5 (a^2b)^{2p},
\end{align}
from \eqref{disc-1}-\eqref{disc-2} specialized to the case $r = 5$.

By Theorems~\ref{modularity-plus}~and~\ref{modularity-minus}, $J^\pm/K$ is modular, which implies the $\rho_{J^\pm,\lambda}$ over all primes $\lambda$ of $K$ form a strictly compatible system. Hence, the $\rho_{J^\pm, \lambda}$ all have isomorphic local Weil-Deligne representations at primes $\Fq \nmid N(\lambda)$, and therefore have the same conductor exponent at $\Fq$ independent of $\lambda \nmid N(\Fq)$.

Let $K_{\Fq_5}$ and $K_{\Fq_2}$ denote the completions of $K$ at the unique primes $\Fq_5$ and $\Fq_2$ of $K$ above $5$ and $2$, with ring of integers $\calO_{\Fq_5}$ and $\calO_{\Fq_2}$, respectively. The notation $\Fq_5$ is used when $K$ is regarded as the base field, in contrast to $\Fr$ introduced earlier, which is used when $K$ is regarded as a coefficient field of a representation.

Let $m$ be a rational prime and $S$ a finite set of (finite) primes of $K$. We denote by
\begin{equation}
  K(S,m) = \left\{ d \in K^*/{K^*}^m : v_{\Fq}(d) \equiv 0 \pmod m \text{ for all } \Fq \not\in S \right\}    
\end{equation}
the \textit{$m$-Selmer group of $K$ with respect to $S$}. When $\zeta_m \in K$, let $K(S,m)^*$ denote the set of characters $\chi: G_K\rightarrow \Z/m\Z$ corresponding to the extensions $K(\sqrt[m]{d})/K$ where $d \in K(S,m)$.

In this section, we prove the following theorem on the conductors of $\brhoJpp$ and a suitable twist of $\brhoJmp$.

\begin{theorem}\label{thm:Serre_level}
Let $p > 5$ be a prime.
\begin{enumerate}
\item Suppose $2 \nmid ab$ and $5 \mid ab$. Then the Serre level of $\brhoJpp$ divides $\Fq_5$ when $5\mid a$ and is equal to $\Fq_5^2$ when $5\mid b$.

\item Suppose $2 \mid a$, $b \equiv -1 \pmod 4$ or $2 \mid c$, $a \equiv -1 \pmod 4$. Then, for some character $\chi_0\in K(S_2, 2)^*$, where $S_2=\lbrace\Fq_2\rbrace$, the Serre level of $\brhoJmptwo$ is equal to $\Fq_2^s \cdot \Fq_5^\epsilon$ where $s=0,1,2$ and $\epsilon=2,3$ if $5 \nmid ab$ and $\epsilon=2$ if $5 \mid ab$. In particular, $s\leq 1$ when $2 \mid a$, $b \equiv -1 \pmod 4$.
\end{enumerate}
For every $\Fp \mid p$ in $K$, the representations $\brhoJpp$ and $\brhoJmptwo$ are finite at $\Fp$.
\end{theorem}

\begin{remark}
The $2$-Selmer group $K(S_2,2)$ can be computed in {\tt Magma} to be
\begin{equation}
  K(S_2,2) = \left\{ 1, -1, -2, 2, \frac{-\sqrt{5} + 1}{2}, \frac{\sqrt{5} - 1}{2}, \sqrt{5} - 1, -\sqrt{5} + 1  \right\}.
\end{equation}
\end{remark}

\begin{remark} The work of Dokchitser-Dokchitser-Maistret-Morgan \cite{DokchitserDokchitserMaistretMorgan19, DDMM-local}
gives a method to compute the conductor of a fixed
hyperelliptic curve over local fields of odd residual
characteristic. Recently in \cite{maistret}, these methods have been applied to compute the conductors at odd primes of Frey hyperelliptic curves for signatures $(p,p,r)$ and $(r,r,p)$.

\begin{remark}
  At the top of \cite[p.\ 425]{DarmonDuke}, it is stated that the particular twists $J_r^+(a,b,c)$ and $J_r^-(a,b,c)$ taken are chosen to be as `little ramified' as possible. We will explain in Remarks~\ref{cond-remark-1} and \ref{cond-remark-3} corrections to which quadratic twist of $J_r^\pm(a,b,c)$ should be taken so the conductor exponent at $\Fq_5$ is minimized, and in Remark~\ref{cond-remark-2} that a different twist of $J_r^-(a,b,c)$ is needed in some cases to minimize the conductor exponent at $\Fq_2$.
\end{remark}

\end{remark}

Let $N(\rho_{J^\pm,\Fp})$ denote the conductor of the $p$-adic representation $\rho_{J^\pm,\Fp}$ (which by definition excludes the primes dividing $p$) and $N(\rho_{J^\pm,\lambda})$ be the conductor of the strictly compatible system of representations $(\rho_{J^\pm,\lambda})$. Away from the primes dividing $p$, the two conductors coincide.

To deal with the conductors of $(\rho_{J^\pm,\lambda})$ at primes $\Fq$ above $q \not= 2, 5$, we first state and prove the following lemma.
\begin{lemma}\label{cond-q}
Let $q \not = 2, 5$ be a prime and $\Fq$ a prime of $K$ above $q$. Then $C^\pm/K$ and $J^\pm/K$ are semistable at $\Fq$. In particular, if $C^\pm/K$ has bad semistable reduction at $\Fq$, then $J^\pm/K$ has multiplicative reduction at $\Fq$.
\end{lemma}

\begin{proof}

We describe the proof for $C^-$ as similar arguments can be used to treat $C^+$ (a transcript of the computations in all cases can be found in \cite{programs}). Suppose $\q$ is a prime of $K$ that does not divide $2$ and $5$. From Lemma \ref{lem:Frey-model} the curve $C^-$ is given by
$$
C^-:~y^2 = x^5 - 5c^2x^3 + 5c^4x - 2(a^p-b^p).
$$

By Proposition \ref{hyperelliptic-good} and \eqref{discriminant-5} if $\Fq \nmid ab$ then $C^-/K$ has good reduction at $\Fq$. If $\Fq \mid a$, the special fiber of the hyperelliptic model $C^-/\calO_K$ at $\Fq$ is given by
\begin{equation}
\label{a-reduce}
y^2 = (x + 2\tilde{c}) (x^2 -\tilde{c}x - \tilde{c}^2)^2,
\end{equation}
where $c$ reduces to the element $\tilde{c}$ in the residue field $\F_\Fq$ of $\Fq$. As $\Delta(x^2 -\tilde{c}x - \tilde{c}^2)=5\tilde{c}^2\neq 0$ in $\F_\Fq$, it follows that $C^-/\calO_K$ satisfies the double root criterion \cite[Lemma 3.7]{BornerBouwWewers2017}. Hence $C^-/K$ has bad semistable reduction at $\Fq$, and therefore $C^-/K$ is semistable at $\Fq$. The remaining case of $\Fq \mid b$ is similar except that \eqref{a-reduce} is changed to
\begin{equation}
  y^2 = (x - 2\tilde{c}) (x^2 +\tilde{c}x - \tilde{c}^2)^2,
\end{equation}

The semistability of $J^\pm/K$ follows from the semistability of $C^\pm/K$ by Theorem \ref{thm:C_J_semistable_reduction}. Using \cite[Proposition 3.1]{BornerBouwWewers2017} and the second part of Theorem~\ref{thm:C_J_semistable_reduction}, we obtain that $C^\pm/K$ having bad semistable reduction at $\Fq$ implies that $J^\pm/K$ has multiplicative reduction at $\Fq$.
\end{proof}

\subsection{The conductor of $(\rho_{J^+,\lambda})$}

We start by first determining an extension of semistable reduction for $C^+/K$ at $\Fq_5$.
\begin{lemma}\label{lem:Cp_mult_Fr}
Suppose $5 \mid ab$. Then the curve $C^+/K_{\Fq_5}$ has potential multiplicative reduction. In particular, the curve has bad semistable reduction over $K_{\Fq_5}$ when $5\mid a$ and bad semistable reduction over $K_{\Fq_5}(5^{1/4})$ when $5\mid b$.
\end{lemma}
\begin{proof}
The initial model of $C^+$ is given by
\begin{equation*}
  y^2 = (x + 2c)(x^5-5 c^2 x^3+5 c^4 x-2 (a^p-b^p)).
\end{equation*}

Assume $5 \mid a$. Let $\tilde{c}$ be the reduction of $c$ modulo $5$. We have one singular point of the special fibre which is $(-2\tilde{c},0)$ by noting that $c \equiv b^p \pmod{5}$. Making the substitutions
\begin{equation*}
x \rightarrow \rf x - 2c, \quad y \rightarrow 5^3 y,
\end{equation*}
we obtain a model over $\calO_{\Fq_5}$ and the equation of the special fibre is given by
\begin{equation}
    y^2 = x^6 + 2\tilde{c}^2x^4 + \tilde{c}^4x^2=x^2(x-2\tilde{c})^2(x + 2\tilde{c})^2.
\end{equation}
The singular points are $(0,0)$ and $(\pm 2\tilde{c},0)$ and each satisfies the double root criterion.


Otherwise $5 \mid b$. We have one singular point of the special fibre which is $(2\tilde{c},0)$ by noting that $c \equiv a^p \pmod{5}$. Making the substitutions 
\begin{equation*}
x \rightarrow \rf x + 2c, \quad y \rightarrow \rf^{5/2} y,
\end{equation*}
we obtain a model over $\calO_{\Fq_5}$ and the equation of the special fibre is given by
\begin{equation}
    y^2 = 4\tilde{c}x^5 + 3\tilde{c}^3x^3 + 4\tilde{c}^5x=-\tilde{c}x(x - 2\tilde{c})^2 (x + 2\tilde{c})^2.
\end{equation}
The singular points are $(\pm 2\tilde{c}, 0)$ and each satisfies the double root criterion. This model is a quadratic twist of the initial model, therefore the curve $C^+$ has bad semistable reduction over the quadratic extension $K_{\Fq_5}(5^{1/4})$ of $K_{\Fq_5}$.
\end{proof}

The next theorem computes the conductor of $(\rho_{J^+,\lambda})$.
\begin{theorem}\label{conductor-plus}
Suppose $2 \nmid ab$ and $5 \mid ab$. Then the conductor $N(\rho_{J^+,\lambda})=\Fq_5^\epsilon\cdot \Fq_{ab}$ where $\Fq_{ab}$ is the square-free ideal of $K$ divisible by the primes $\Fq \not= \Fq_5$ and $\Fq\mid ab$ and
\begin{equation*}
\epsilon = \begin{cases}
    1, & 5\mid a,\\
    2, & 5\mid b.
\end{cases}
\end{equation*}
\end{theorem}

\begin{proof}
By the proof of \cite[Proposition 1.15]{DarmonDuke}, $C^+/K$ has good reduction at the prime $\Fq_2$. Applying Lemma \ref{cond-q}, Lemma \ref{lem:Cp_mult_Fr} and Theorem \ref{thm:C_J_semistable_reduction} shows that $J^+/K$ has multiplicative reduction at the primes $\Fq \mid \Fq_5 \cdot \Fq_{ab}$ and good reduction for the other primes $\Fq$. For the primes $\Fq$ of potential multiplicative reduction, we conclude by Grothendieck's inertial criterion \cite[Expos\'e 9]{Raynaud} that the compatible system $(\rho_{J^+,\lambda} \mid_{D_\Fq})$, where $\lambda$ runs through the primes of $K$ not lying above $N(\Fq)$, is special type with respect to a character $\chi$. In particular, $\chi$ is trivial when $\Fq\neq\Fq_5$ or $\Fq=\Fq_5$ and $5\mid a$, and the $\chi$ is a totally ramified quadratic character with respect to the extension $K_{\Fq_5}(5^{1/4})/K_{\Fq_5}$ when $5\mid b$. In the former case the conductor is $\Fq_5^1$ at $\Fq_5$ while in the latter case is $\Fq_5^2$.
\end{proof}

\begin{remark}\label{cond-remark-1}
We point out that the statement in \cite[Theorem 3.5(1)]{DarmonDuke} for $J_r^+(a,b,c)$ is only true up to a quadratic twist when $r \mid b$. For instance, it would imply that $\rhobar_{J^+,\Fp}$ has conductor $\le 1$ at $\Fq_5$ when $5 \mid b$, which is not true: we have shown that $\rho_{J^+,\Fp}$ has special type and conductor exponent $2$ at $\Fq_5$ so using \cite[Theorem 1.5]{Jarvis} we conclude $\rhobar_{J^+,\Fp}$ has conductor exponent $2$ at $\Fq_5$. The twist of $\rhobar_{J^+,\Fp}$ by the character $\chi$ has conductor exponent $\le 1$ at $\Fq_5$.
\end{remark}

\subsection{The conductor of $(\rho_{J^-,\lambda})$}

In this section we compute the conductor of $(\rho_{J^-,\lambda} \otimes \chi_0)$ for some choice of $\chi_0 \in K(S_2, 2)^*$ which ``minimizes'' this conductor.

\begin{theorem}\label{conductor-minus}
Suppose $2 \mid a$, $b \equiv -1 \pmod 4$ or $2 \mid c$, $a \equiv -1 \pmod 4$, and $p > 5$. Then, for some choice of $\chi_0$, the conductor $N(\rho_{J^-,\lambda} \otimes \chi_0)$ is of the form $\Fq_2^s \cdot \Fq_5^\epsilon\cdot \Fq_{ab}$ where $\epsilon=2,3$ if $5 \nmid ab$ and $\epsilon = 2$ if $5 \mid ab$, 
\begin{equation*}
    s = \begin{cases}
        1, & \text{if } 2\mid a,\\
        1\text{ or } 2, & \text{if } 2\mid c.
    \end{cases}
\end{equation*}
and $\Fq_{ab}$ is the square-free ideal of $K$ divisible by the primes $\Fq\mid ab$ and $\Fq\nmid 10$.
\end{theorem}

\begin{proof}
This will follow from Lemma~\ref{cond-q} (for the conductor exponent at primes dividing $\Fq_{ab}$ in the same way as the last paragraph of the proof of Theorem~\ref{conductor-plus}), Proposition \ref{cond-2B} (for the conductor exponent at $\Fq_2$), and Proposition \ref{cond-5} (for the conductor exponent at $\Fq_5$) if $5 \nmid ab$ and Proposition \ref{prop:conductor_at_5_J_minus_5_mid_ab} if $5 \mid ab$.
\end{proof}


\subsubsection{The conductor of $(\rho_{J^-,\lambda})$ at $\Fq_2$}

We have used \eqref{L-isom} to ``propagate modularity'' from $\mathcal{L}/K$ to $J^-/K$ in the previous section. In this section, we will use \eqref{L-isom} to ``propagate conductors'' from $\mathcal{L}/K$ to $J^-/K$.

Consider
\begin{equation*}
    E :~y^2 = x (x + a^p)(x - b^p),
\end{equation*}
which is a certain quadratic twist of the Legendre family for $t = a^p/c^5$. The lemma below determines the conductor of $E/K$ at $\Fq_2$.
\begin{lemma}\label{legendre-info}
  Suppose $p > 3$ and
  \begin{enumerate}
      \item $2 \mid a$, $b \equiv -1 \pmod 4$, or
      \item $2 \mid c$, $a \equiv -1 \pmod 4$.
  \end{enumerate}
  Then, the conductor at $\Fq_2$ of the elliptic curve $E/K$ given by
  \begin{equation}
      E: y^2 = x(x + a^p)(x-b^p)
  \end{equation}
  is $\Fq_2$. Furthermore, $\rhobar_{E,5}$ has conductor at $\Fq_2$ dividing $\Fq_2$.
\end{lemma}

\begin{proof}
This can be proven using Tate's algorithm, for instance in \cite[\S 2.2]{DDT}, it is shown that if
\begin{equation}
     A + B  = C
\end{equation}
where $(A,B,C) \in \Z^3$ are coprime, $A \equiv -1 \pmod 4$, $B \equiv 0 \pmod {16}$, then the elliptic curve
\begin{equation}
    y^2 = x (x - A)(x + B)
\end{equation}
is semistable over $\Q$ with conductor at $2$ equal to $2$, and 
\begin{equation}
    \Delta_E = 2^{-8} A^2 B^2 C^2.
\end{equation}
If $2 \mid a$, $b \equiv -1 \pmod 4$ (resp.\ $2 \mid c$, $a \equiv -1 \pmod 4$) then applying the above result with $A = b^p, B = a^p$ (resp.\ $A = a^p, B = - c^5$) yields the conductor at $2$ (and hence at $\Fq_2$). The minimal discriminant of $E$ is given by $\Delta_E = 2^{-8} (ab)^{2p}c^{10}$ in both cases.
\end{proof}  

\begin{remark}
\label{not-lower-at-2}
Note $\rhobar_{E,5}$ has conductor at $\Fq_2$ equal to $\Fq_2$ if and only if  $5 \nmid v_2(\Delta_E) = v_2((ab)^{2p} c^{10}) - 8$.
\end{remark}

Our method of ``propagating the conductor at $\Fq_2$'' of $\mathcal{L}/K$ to $J^-/K$ requires that we know the inertial type of $J^-/K$ at $\Fq_2$ which is provided by the following result.

\begin{proposition}\label{potmul2}
Suppose $2 \mid a$, $b \equiv -1 \pmod 4$. Then $C^-/K_{\Fq_2}$ has potentially bad semistable reduction and $\rho_{J^-,\lambda} \mid_{I_{K_{\Fq_2}}}$ has special inertial type with $\lambda\nmid N(\Fq_2)$.
\end{proposition}

\begin{proof}
The Igusa invariants of $C^-$ are given by 
\begin{align*}
    J_2 & = 700c^4,\\
    J _4 & = 13750c^8, \\ 
    J_6 & = 1280000a^{2p}c^2 - 1280000a^pc^7 + 312500c^{12},\\
    J_8 & = 224000000a^{2p}c^6 - 224000000a^pc^{11} + 7421875c^{16},\\
    J_{10} & = 204800000a^{4p} - 409600000a^{3p}c^5 + 204800000a^{2p}c^{10}
\end{align*}

According to \cite[Th\'eor\`eme 1 (I)]{Liu93} the curve $C^-/K_{\Fq_2}$ has potential good reduction if and only if $J_{2i}^5/J_{10}^i$ is integral for all $1\leq i\leq 5$. It holds that
\begin{equation*}
    \frac{J_2^5}{J_{10}} = \frac{5^57^5c^{20}}{2^6a^{2p}(a^p - c^5)^2},
\end{equation*}
which is not integral. Therefore, $C^-/K_{\Fq_2}$ has potential bad semistable reduction. 
\end{proof}

We need the following lemmas to determine a complete set of twists needed to minimize the conductor of $(\rho_{J^-,\lambda} \otimes \chi_0)$ at $\Fq_2$.
\begin{lemma}\label{lem:d_congrunece_p2_square}
Let $K$ be an unramified extension of $\Q_2$ with uniformizer $\pi$ and ring of integers $\calO_K$. Suppose $d \in \OK^*$ satisfies $d \equiv 1 \pmod{\pi^2}$. Then $L = K(\sqrt{d})$ is unramified over $K$. 
\end{lemma}

\begin{proof}
Let $\alpha = \sqrt{d}$ and $\beta = \frac{1+\alpha}{2}$. Then $N_{L/K}(\beta) = \frac{1 - \alpha^2}{4} \in \calO_K$ and $T_{L/K}(\beta) = 1 \in \calO_K$. Hence, $\beta \in \calO_L$. Now,
\begin{align}
    \beta^2 & = \frac{1 + 2 \alpha + \alpha^2}{4} \\
    & = \frac{1 + d}{4} + \frac{\alpha}{2}.
\end{align}

The relative discriminant of $\calO_K[\beta] = \calO_K + \calO_K \beta$ over $\calO_K$ is the determinant of
\begin{equation*}
    \begin{pmatrix}
    2 & 1 \\
    1 & \frac{1+d}{2}
    \end{pmatrix}
\end{equation*}
which is $d$. It now follows that $\calO_L = \calO_K[\beta]$ and the relative discriminant of $L/K$ is $d \calO_K$ which is not divisible by $\pi$. Hence, $L/K$ is unramified.
\end{proof}

\begin{lemma}\label{lem:q2_unramified_character}
Let $\chi:G_{K}\rightarrow\lbrace\pm 1\rbrace$ be a quadratic or the trivial character. Then there exists a character $\chi_0\in K(S_2, 2)^*$ where $S_2=\lbrace\Fq_2\rbrace$ such that $\chi\chi_0$ is unramified at $\Fq_2$.
\end{lemma}

\begin{proof}
If $\chi$ is the trivial character then we choose $\chi_0$ to be the trivial character as well. We now assume that $\chi$ is a non-trivial character and let $L=K(\sqrt{d})$ be the corresponding quadratic extension, where $d\in\OK^*$ is square-free.

Since the class number of $K$ is $1$, using the idelic definition of class group \cite[Chapter VI, Proposition 1.3]{neukirch}, there exists a $u \in K^\times$ such that
\begin{equation}
\label{lifting}
    d u^{-1} \in \hat \calO_K^\times,
\end{equation}
where $\hat \calO_K^\times = \prod_{v \text{ finite}} \calO_{K_v}^\times$. Let $\calO_K^\times = \langle -1, \frac{\sqrt{5}+1}{2} \rangle$. It can be checked that reduction map $\calO_K^\times$ to $\left( \calO_K/\Fq_2^2 \right)^\times$ is surjective. Hence, without loss of generality, we may assume that $u \equiv 1 \pmod{\Fq_2^2}$.

Let $d_0 = d u^{-1}$, then $d_0 \in K(S_2, 2)$ since $K(\sqrt{d_0})$ is unramified outside of $\Fq_2$ by \eqref{lifting}. Let $\chi_0\in K(S_2, 2)^*$ and $\chi_u$ be the characters corresponding to $K(\sqrt{d_0})$ and $K(\sqrt{u})$, respectively. From $d_0 = d u^{-1}$, we have the relation $\chi_u = \chi \chi_0^{-1} = \chi \chi_0$. As $u \in \calO_K$ and $u \equiv 1 \pmod{\Fq_2^2}$, by Lemma \ref{lem:d_congrunece_p2_square} we have that $\chi_u$ is unramified at $\Fq_2$, therefore $\chi\chi_0$ is unramified at $\Fq_2$.
\end{proof}

The next proposition propagates the conductor at $\Fq_2$ of $\rho_{\mathcal{L},5}$ to $\rho_{J^-,\Fr} \otimes \chi_0$ for some choice of $\chi_0 \in K(S_2,2)^*$ that minimizes this conductor.
\begin{proposition}\label{cond-2B-Fr}
Suppose $2\mid a$, $b \equiv -1 \pmod 4$ or $2 \mid c$, $a \equiv -1 \pmod 4$, and $p > 5$. Then there is a character $\chi_0\in K(S_2, 2)^*$, where $S_2=\lbrace\Fq_2\rbrace$, such that the conductor at $\Fq_2$ of $\rho_{J^-,\Fr}\otimes\chi_0$ is equal to $\Fq_2^s$ where 
\begin{equation*}
    s = \begin{cases}
        1, & \text{if } 2\mid a,\\
        1\text{ or } 2, & \text{if } 2\mid c.
    \end{cases}
\end{equation*}
\end{proposition}

\begin{proof}
To the solution $(a,b,c)$ we attach the elliptic curve
$$
E:y^2 = x(x + a^p)(x-b^p).
$$
The elliptic curve $E$ is a quadratic twist of $\mathcal{L}$ in \eqref{L-curve} and hence by \eqref{L-isom} we have that
\begin{equation}
\label{E-isom}
    \rhobar_{J^-,\Fr} \cong \rhobar_{E,5} \otimes \chi,
\end{equation}
for some quadratic character $\chi$ of $G_K$. By Lemma \ref{lem:q2_unramified_character} there exists a character $\chi_0\in K(S_2, 2)^*$ such that $\chi\chi_0$ is unramified at $\Fq_2$. Applying Lemma~\ref{legendre-info}, the conductor at $\Fq_2$ of $\rhobar_{J^-,\Fr}\otimes\chi_0\simeq\rhobar_{E,5} \otimes \chi \chi_0$ divides $\Fq_2$.

The representation $\rhobar_{J^-,\Fr}$ is absolutely irreducible by Proposition \ref{irreducible-minus-5} and modular by Theorem \ref{modularity-minus}. Therefore, as we explain below we can conclude that $\rho_{J^-,\Fr}\otimes\chi_0$ has conductor $\Fq_2^s$ where $s = 1, 2$.

If $2 \mid a$, $b \equiv -1 \pmod 4$, then by Proposition~\ref{potmul2}, $\rho_{J^-,\Fr} \otimes \chi_0$ has special type. If $s \ge 2$, there is degeneration of the conductor under reduction and we are in case (2) of \cite[Theorem 1.5]{Jarvis}. But this cannot occur as $N_{K/\Q}(\Fq_2) - 1 = 3 < 5$. Hence, $s \le 1$, and by Proposition~\ref{potmul2} again, $s = 1$.

If $2 \mid c$, $a \equiv -1 \pmod 4$, then by the proof of Lemma~\ref{legendre-info} and Remark~\ref{not-lower-at-2}, $\rhobar_{E,5}$ has conductor at $\Fq_2$ equal to $\Fq_2$ as $5 \nmid v_2(\Delta_E)$. Hence, we cannot have $s = 0$. Thus, we are either have $s = 1$ or we are in cases (2) -- (4) of \cite[Theorem 1.5]{Jarvis}. As $N_{K/\Q}(\Fq_2) - 1 = 3 < 5$, the only possibility is case (4) and $s = 2$.
\end{proof}

\begin{remark}
\label{cond-remark-2}
We point out the statement in \cite[Theorem 1.17]{DarmonDuke} for $J_r^-(a,b,c)$ is only true up to a quadratic twist if $2 \mid ab$. For instance, it would imply that $\rhobar_{J^-,\Fp}$ has conductor exponent $\le 1$ at $\Fq_2$ when $2 \mid ab$, which is not true: Consider $a = 2$, $c = 1$ so that $c^5 - a^7 = b^7$ for some $b \in \Z_2$. Using \cite{DokchitserDoris19}, which applies for a fixed hyperelliptic curve of genus $2$, we find that the conductor exponent at $\Fq_2$ of $\rho_{J^-,\Fp}$ is $4$. We have shown by Proposition~\ref{potmul2} that $\rho_{J^-,\Fp}$ has special type at $\Fq_2$ and by Proposition~\ref{cond-2B-Fr} some quadratic twist of $\rho_{J^-,\Fp}$ has conductor exponent $1$ at $\Fq_2$. We conclude from \cite[Theorem 1.5]{Jarvis} that the conductor exponent of $\rhobar_{J^-,\Fp}$ at $\Fq_2$ is $4$. Unlike Remarks~\ref{cond-remark-1} and \ref{cond-remark-2}, we do not have a precise description of which twist brings the conductor exponent of $\rhobar_{J^-,\Fp}$ to $1$.
\end{remark}

\begin{remark}
There is less known about calculating conductors of hyperelliptic curves at primes above $2$, in particular the recent methods of \cite{DokchitserDokchitserMaistretMorgan19, DDMM-local, maistret} do not cover this case. Our method of ``propagating'' conductor computations from the Legendre elliptic curve, though slightly weaker than an exact computation, avoids these difficulties and suffices for the purposes of Diophantine applications. 
\end{remark}



The final proposition which computes the conductor at $\Fq_2$ of $(\rho_{J^-,\lambda} \otimes \chi_0)$ now follows from compatibility.

\begin{proposition}\label{cond-2B}
Suppose $2\mid a$, $b \equiv -1 \pmod 4$ or $2 \mid c$, $a \equiv -1 \pmod 4$, and $p > 5$. Let $\chi_0$ be a character as in Proposition \ref{cond-2B-Fr}. Then the conductor at $\Fq_2$ of $(\rho_{J^-,\lambda} \otimes \chi_0)$ is equal to $\Fq_2^s$ where 
\begin{equation*}
    s = \begin{cases}
        1, & \text{if } 2\mid a,\\
        1\text{ or } 2, & \text{if } 2\mid c.
    \end{cases}
\end{equation*}
\end{proposition}

\begin{proof}
The conductor at $\Fq_2$ of the compatible system $(\rho_{J^-,\lambda} \otimes \chi_0)$ is independent of $\lambda$ for $\lambda \nmid N(\Fq_2)$ as we have shown $J^-/K$ is modular, so by Proposition \ref{cond-2B-Fr} we have the conclusion.
\end{proof}

\begin{remark}
As $\chi_0$ is unramified outside of $\Fq_2$, the conductors of $(\rho_{J^-,\lambda})$ and $(\rho_{J^-,\lambda} \otimes \chi_0)$ are equal away from $\Fq_2$.
\end{remark}

\subsubsection{The conductor of $(\rho_{J^-,\lambda})$ at $\Fq_5$}

We state and prove two propositions, one which determines an extension of semistable reduction for $C^-/K_{\Fq_5}$ and another which uses this to determine the conductor of $(\rho_{J^-,\lambda})$ at $\Fq_5$.

\begin{lemma}\label{lem:semistable_at_5_J_minus_5_mid_ab}
Suppose $5\mid ab$. Then the curve $C^-/K_{\Fq_5}$ has potential bad semistable reduction. In particular, the curve has bad semistable reduction over $K_{\Fq_5}(5^{1/4})$.
\end{lemma}

\begin{proof}
The initial model of $C^-$ is given by
\begin{equation*}
    y^2 = x^5 - 5c^2x^3 + 5c^4x - 2(a^p - b^p).
\end{equation*}


Assume $5\mid a$. Let $\tilde{c}$ be the reduction of $c$ modulo $5$. We have one singular point on the special fibre which is $(-2\tilde{c}, 0)$. Making the substitution 
\begin{equation*}
x\rightarrow \rf x - 2c, \quad y\rightarrow \rf^{5/2}y,
\end{equation*}
we obtain a model over $\calO_{\Fq_5}$ and the equation of the special fibre is given by
\begin{equation}
    y^2 = x^5 + 2\tilde{c}^2x^3 + \tilde{c}^4x=x(x-2\tilde{c})^2(x+2\tilde{c})^2.
\end{equation}
The singular points are $(\pm 2\tilde{c},0)$ which both satisfy the double root criterion. The new model is a quadratic twist of the initial model, thus $C^-$ attains multiplicative reduction over the extension $K_{\Fq_5}(5^{1/4})$.

Otherwise $5\mid b$. We have one singular point of the special fibre which is the point $(2\tilde{c}, 0)$. Using the substitution 
\begin{equation*}
x\rightarrow \rf x + 2c, \quad y\rightarrow \rf^{5/2}y,
\end{equation*}
we obtain a model over $\calO_{\Fq_5}$ and the equation of the special fibre is given by
\begin{equation}
    y^2 = x^5 + 2\tilde{c}^2x^3 + \tilde{c}^4x=x(x - 2\tilde{c})^2 (x + 2\tilde{c})^2.
\end{equation}
The singular points are $(\pm 2\tilde{c}, 0)$ and each satisfies the double root criterion. The new model is a quadratic twist of the initial model again, thus $C^-$ attains multiplicative reduction over the extension $K_{\Fq_5}(5^{1/4})$.
\end{proof}

\begin{proposition}\label{prop:conductor_at_5_J_minus_5_mid_ab}
Suppose $5\mid ab$. Then the conductor at $\Fq_5$ of $(\rho_{J^-,\lambda})$ is equal to $\Fq_5^2$.
\end{proposition}

\begin{proof}
The proof is similar to the proof of Theorem \ref{conductor-plus}. By Lemma \ref{lem:semistable_at_5_J_minus_5_mid_ab} the curve $C^-$ attains multiplicative reduction over $K_{\Fq_5}(5^{1/4})$. Hence, we conclude by Grothendieck's inertial criterion \cite[Expos\'e 9]{Raynaud} that the compatible system $(\rho_{J^+,\lambda} \mid_{D_{\Fq_5}})$, where $\lambda$ runs through the primes of $K$ not lying above $N(\Fq)$, is special type with respect to the character $\chi$ associated to the extension $K_{\Fq_5}(5^{1/4})/K_{\Fq_5}$. We conclude that the conductor at $\Fq_5$ is $\Fq_5^2$.
\end{proof}

\begin{remark}
\label{cond-remark-3}
We point out that the statement in \cite[Theorem 3.5(1)]{DarmonDuke} for $J_r^-(a,b,c)$ is only true up to a quadratic twist for $r \mid ab$. For instance, it would imply that $\rhobar_{J^-,\Fp}$ has conductor exponent $\le 1$ at $\Fq_5$ when $5 \mid ab$, which is not true: we have shown that $\rho_{J^-,\Fp}$ has special type and conductor exponent $2$ at $\Fq_5$ so using \cite[Theorem 1.5]{Jarvis} we conclude $\rhobar_{J^-,\Fp}$ has conductor exponent $2$ at $\Fq_5$. The twist of $\rhobar_{J^-,\Fp}$ by the character $\chi$ has conductor exponent $\le 1$ at $\Fq_5$.
\end{remark}

\begin{proposition}\label{semistable_at_5_J_minus}
Suppose $5 \nmid ab$. Let $M/K_{\Fq_5}$ be a totally ramified extension of degree $4$ and 
\begin{equation*}
  \phi(x) = x^5-5 c^2 x^3+5 c^4 x-2 (a^p-b^p).
\end{equation*}  
Let $d(a,c)$ be the constant term of $g(x) = \phi(x - a^p - 2c)$.
\begin{enumerate}[(i)]
\item If $v_5(d(a,c))\geq 2$, then $C^-/K_{\Fq_5}$ attains good reduction over the extension $M/K_{\Fq_5}$. 
\item If $v_5(d(a,c)) = 1$, let $a_0, b_0, c_0 \in \Z$ be the least non-negative residues of $a, b, c$ modulo $25$, respectively. Then, there is an extension $L = L_{(a_0,b_0,c_0)}$ of $K_{\Fq_5}$ depending on $(a_0,b_0,c_0)$ such that $L/K_{\Fq_5}$ is a degree $20$ totally ramified extension and $C^-/K_{\Fq_5}$ attains good reduction over $L/K_{\Fq_5}$. 
\end{enumerate}

Moreover, these extensions are minimal with respect to ramification index.
\end{proposition}

\begin{proof}
The initial model for $C^-$ is given by the equation $y^2 = \phi(x)$. The special fibre of $C^-$ has a singular point which is $(-\tilde{a}^p - 2\tilde{c},0)$ where $\tilde{a}, \tilde{c}$ are the reductions of $a,c$ modulo $5$. Hence we consider the transformation $x \rightarrow x - a^p - 2 c$ which results in the hyperelliptic model $y^2 = g(x) = \phi(x - a^p - 2 c)$ of $C^-$ which has valuation vectors over $M$ given by
\begin{align}
\label{val-1} (8 ,8,\geq 8,8, \geq 8,0,\infty)\qquad (\infty,\infty,\infty,\infty), & & \text{if }v_5(d(a,c))=1,\\
\label{val-2} (\geq 16 ,8,\geq 8,8, \geq 8,0,\infty)\qquad (\infty,\infty,\infty,\infty), & & \text{if }v_5(d(a,c))\geq 2.
\end{align}
These valuations are determined by expanding the coefficients of $g(x)$ as power series in $a,c$.

\textbf{Case $v_5(d(a,c)) \ge 2$:} From \eqref{discriminant-5} we have that $v_M(\Delta(C^-))=40$. From \eqref{val-2} and \eqref{odd-discrim-trans}, the substitution $x \rightarrow \pi_M^2 x$, $y \rightarrow \pi_M^5 y$ gives us a model with good reduction over $M$, where $\pi_M$ is a uniformizer of $M$.

The extension $M/K_{\Fq_5}$ has minimal ramification index with the property that $C^-/M$ has good reduction: Suppose $C^-$ attains good reduction over an extension $F/K_{\Fq_5}$ with ring of integers $\calO_F$. As $C^-$ attains good reduction over $F$, it has an odd degree hyperelliptic model over $\mathcal O_F$ by Proposition~\ref{hyperelliptic-good}. Since $v_{K_{\Fq_5}}(\Delta(C^-)) = 10$ from the hypothesis $5 \nmid ab$, it follows that the ramification index of $F/K_{\Fq_5}$ must be divisible by $4$ from \eqref{odd-discrim-trans}.

\textbf{Case $v_5(d(a,c)) = 1$:} Let $\phi_0(x) = x^5-5 c_0^2 x^3+5 c_0^4 x-2 (a_0^p-b_0^p)$ and $g_0(x) = \phi_0(x - a_0^p - 2c_0)$. The coefficients of $g(x)$ lie in $\Q_5$ and $M/\Q_5$ is a degree $8$ totally ramified extension so Eisenstein's criterion holds for $g(x)$ over $\Q_5$. As $g(x)\equiv g_0(x)\pmod{25}$, $g_0(x)$ also satisfies Eisenstein's criterion over $\Q_5$. 

Let $\theta_0$ be a root of $g_0(x)$ and $L = M(\theta_0)$. Then $L/K_{\Fq_5}$ is a $20$ degree totally ramified extension. Making the substitution $x\rightarrow x + \theta_0$ to the model $y^2 = g(x)$ of $C^-$, we obtain a model with valuation vectors over $L$
\begin{equation}
\label{val-3} (\geq 80,40,\geq 40, 40,\geq 40,0,\infty)\qquad (\infty,\infty,\infty,\infty).
\end{equation}
Again from \eqref{discriminant-5} we have that $v_L(\Delta(C^-))=200$. From \eqref{val-3} and \eqref{odd-discrim-trans}, the substitution $x\rightarrow \pi_L^{10}x$, $y \rightarrow \pi_L^{25}y$ gives us a model with good reduction over $L$, where $\pi_L$ is a uniformizer of $L$.

The extension $L/K_{\Fq_5}$ has minimal ramification index with the property that $C^-/L$ has good reduction: Suppose $C^-$ attains good reduction over an extension $F/K_{\Fq_5}$. Similar to the previous case, it follows that $4 \mid e(F/K_{\Fq_5})$.

We know from \cite{SerreTate} that $[K^{\text{nr}}_{\Fq_5}(J^-[2]):K^{\text{nr}}_{\Fq_5}]\mid [F\cdot K^{\text{nr}}_{\Fq_5}:K^{\text{nr}}_{\Fq_5}]$, where $K_{\Fq_5}^{\text{nr}}$ is the maximal unramified extension of $K_{\Fq_5}$. As $g(x)$ is an Eisenstein polynomial over $\Q_5$, the extension $\Q_5(\theta)/\Q_5$ is a totally ramified extension of degree $5$, where $\theta$ is a root of $g(x)$. This in turn implies the extension $K_{\Fq_5}(\theta)/K_{\Fq_5}$ is a totally ramified extension of degree $5$. Since $K_{\Fq_5}^\text{nr}(\theta) \subset K_{\Fq_5}^\text{nr}(J^-[2])$ it follows that $5\mid [K^{\text{nr}}_{\Fq_5}(J^-[2]):K^{\text{nr}}_{\Fq_5}]$, and hence $5\mid [F\cdot K^{\text{nr}}_{\Fq_5}:K^{\text{nr}}_{\Fq_5}]$ so $5 \mid e(F/K_{\Fq_5})$. We have thus shown $20\mid e(F/K_{\Fq_5})$.
\end{proof}

By considering the Newton polygon of $g(x)$, it can be shown that the condition $v_5(d(a,c)) \ge 2$ is equivalent to the polynomial $g(x)$, and hence $\phi(x)$, being reducible in $\Q_5[x]$

\begin{proposition}\label{cond-5}
Suppose $5 \nmid ab$. Then the conductor of $(\rho_{J^-,\lambda})$ at $\Fq_5$ is $\Fq_5^\epsilon$ where
\begin{equation*}
    \epsilon = \begin{cases}
        2, & \text{if } v_5(d(a,c)) \ge 2,\\
        3, & \text{if } v_5(d(a,c)) = 1.
    \end{cases}
\end{equation*}
\end{proposition}

\begin{proof}
Let $F = M$ or $L$ as above. The field cut out by $\rho_{J^-, \lambda} \mid_{I_{K_{\Fq_5}}}$ corresponds to the extension $F \cdot K^\text{nr}_{\Fq_5}/K^\text{nr}_{\Fq_5}$ when $\lambda\nmid N(\Fq_5)$. This follows from \cite{SerreTate} and the fact that 
\begin{equation*}
\rho_{J^-,\ell} \simeq \mathop{\oplus}_{\lambda \mid \ell} \rho_{J^-, \lambda},
\end{equation*}
where $\rho_{J^-,\ell}$ is the representation of $G_K$ on the full Tate module $T_\ell(J^-)$ and $\ell\neq 5$ a rational prime. Since $J^-/K$ is modular, the $\rho_{J^-,\lambda}$ are strictly compatible and hence have the same associated Weil-Deligne representation. It follows that the field cut out by $\rho_{J^-,\ell} \mid_{I_{K_{\Fq_5}}}$ is the same as the field cut out by any one of the $\rho_{J^-,\lambda} \mid_{I_{K_{\Fq_5}}}$.

To ease notation set $K = K_{\Fq_5}$, $G_K = G_{K_{\Fq_5}}$, $I_K = I_{K_{\Fq_5}}$ from here on until the end of this proof. There are finitely many possibilities for $F$. For each possible choice, we may compute the conductor of $\rho_{J^-,\lambda} \mid_{I_{K}}$ in the following way.

Let $\rho : G_{K} \rightarrow \GL_2(\C)$ be a continuous Galois representations that factors through the Galois group of the normal closure $\tilde{F}/K$ of $F/K$ with inertia group order equal to $4$ if $F = M$ and $20$ if $F = L$. 

Using \texttt{Magma's} function \textit{GaloisRepresentations}, we compute all irreducible Galois representations of $G_{K}$ that factor through the Galois group of the normal closure $\tilde{F}/K$ of $F/K$ with inertia group order equal to $4$ if $F = M$ and $20$ if $F = L$. Each such representation has dimension $1$ and conductor $\Fq_5$ when $F = M$, and dimension $2$  and conductor $\Fq_5^3$ when $F = L$.

Therefore, the conductor of $\rho$ is equal to $\Fq_5^2$ when $F = M$ as the inertial type of $\rho$ in this case must be a principal series with characters that are inverses of each other on inertia; the conductor of $\rho$ is equal to $\Fq_5^3$ when $F = L$ and the inertial type in this case is supercuspidal.



The above \texttt{Magma} command computes the conductor of Artin representations $\rho$ of $G_{K}$ whereas we want to compute the conductor of the $\lambda$-adic representation $\rho_{J^-,\lambda} \mid_{G_{K}}$. To make the link between the two types of representations, we show that there is a continuous $2$-dimensional representation $\rho : G_{K} \rightarrow \GL_2(\C)$ factoring through $\tilde{F}/K$ associated to $\rho_{J^-,\lambda} \mid_{G_{K}}$ and having the same conductor. 

Let 
\begin{equation*}
(\rho_{J^-}, N) = \text{WD}(\rho_{J^-,\lambda} \mid_{G_{K}}),
\end{equation*}
be the Weil-Deligne representation attached to $\rho_{J^-,\lambda} \mid_{G_{K}}$. Since $N = 0$ in this case, $\rho_{J^-}$ is a representation of $W_{K}$ and we know that it is parametrized by quasicharacters $\chi_1, \chi_2$ of $W_{K}$ (principal series) or a quasicharacter $\psi$ of $W_{M_\nu}$ where $M_\nu/K$ is a quadratic extension (supercuspidal, odd residual characteristic).

Each quasicharacter $\chi_i$, $\psi$ can be multiplied by a (real) power of the (unramified) quasicharacter $| \cdot |_{\Fq_5}$, $| \cdot |_\nu$ to obtain a character $\tilde{\chi}_i$, $\tilde{\psi}$ with finite image, respectively, and whose conductor is the same. Using $\tilde{\chi}_1$, $\tilde{\chi}_2$ in the principal series case, and $\tilde{\psi}$ in the supercuspidal case, we obtain a representation $\tilde{\rho}$ of $W_{K}$ with the same conductor as $\rho_{J^-}$ since by construction $\tilde{\rho} \mid_{I_K} \simeq \rho_{J^-} \mid_{I_K}$.

By \cite[Theorem (4.2.1)]{Tate-Corvallis} the representation $\tilde{\rho}$ gives a continuous representation $W_{K} \rightarrow \GL_2(\overline{\Q}_p)$ with finite image, and hence $\tilde{\rho}$ can be regarded as a continuous representation $W_{K} \rightarrow \GL_2(\C)$. As it is explained in \cite[Section 2.2]{Tate-Corvallis} we can extend $\tilde{\rho}$ to a continuous representation $\rho : G_{K} \rightarrow \GL_2(\C)$. 


Since the restrictions $\tilde{\rho} \mid_{I_{K}} \simeq \rho_{J^-} \mid_{I_{K}}$ cut out the same field, we deduce that $\tilde{F} \cdot K^\text{nr}/K^\text{nr} = F \cdot K^\text{nr}/K^\text{nr}$ and both restrictions factor faithfully through $I_{\tilde{F}/K} = I_K \cdot G_{\tilde{F}}/G_{\tilde{F}}$. In particular, $\# \rho(I_K)$ is equal to $4$ if $F = M$ and $20$ if $F = L$.

We have determined the conductor of every possible Artin representation $\rho$ as above, and therefore deduce from this computation that the conductor of $\rho_{J^-,\lambda} \mid_{G_{K}}$ is equal to $\Fq_5^2$ when $F = M$, and $\Fq_5^3$ when $F = L$.
\end{proof}

\begin{remark}
With the notation in the last proof above, we may obtain an upper bound on the conductor exponent at $\Fq_5$ using the valuation of the relative different of $F/K$ and \cite[(18)]{LRS}. Using \texttt{Magma} we get that the bound on the conductor exponent is $\le 3$ in case $F = L$ and $\le 2$ in case $F = M$. The inertial type of $\rho_{J^-,\lambda}$ cannot be unramified nor special so the conductor exponent is $\ge 2$. When $F = M$, we deduce the conductor exponent is exactly $2$. When $F = L$, we note that $F$ is wildly ramified. If the inertial type is principal series with characters $\chi_1, \chi_2$, then both $\chi_1$ and $\chi_2$ are wildly ramified, since we know that $\chi_1 \chi_2 = 1$ on inertia. Thus, the conductor exponent of $\chi_1$ and $\chi_2$ are $\ge 2$. It follows that the conductor exponent of $\rho_{J^-,\lambda}$ is $\ge 4$ by the conductor formula for principal series, a contradiction. Hence, the inertial type is non-exceptional supercuspidal induced from a character $\chi$ on a quadratic extension $F'/F$. Since the wild inertia group lies in $G_{F'}$, it follows that $\chi$ is wildly ramified. Hence, the conductor exponent of $\chi$ is $\ge 2$ and the conductor exponent of $\rho_{J^-,\lambda}$ is $\ge 3$ by the conductor formula for non-exceptional supercuspidal type (in fact, we are in the case when $F'/F$ is ramified as the other case would predict a conductor exponent that is $\ge 4$). In summary, we deduce that the conductor exponent is exactly $3$ when $F = L$.
\end{remark}

\begin{remark}
We also point out the statement in \cite[Proposition 1.16]{DarmonDuke} which covers general $r$, but only gives a bound on the conductor exponent at $\Fr$.
\end{remark}

\subsection{Proof of Theorem~\ref{thm:Serre_level}}

For primes $\Fq$ of $K$ not dividing $10$, $\brhoJpmp$ is unramified at $\Fq \nmid p$, and finite at $\Fq$ if $\Fq \mid p$ by \cite[Proposition 1.15]{DarmonDuke}. The same is true for $\rhobar_{J^-,\Fp} \otimes \chi_0$ as $\chi_0$ is unramified at $\Fq \nmid 10 p$.

By Theorem \ref{conductor-plus} for $(\rho_{J^+,\lambda})$ we have that 
\begin{equation*}
  N(\rho_{J^+,\lambda})=\Fq_5^\epsilon\cdot \prod_{\Fq \not= \Fq_5, \Fq \mid ab} \Fq,
\end{equation*}
where $\epsilon = 1$ when $5\mid a$ and $\epsilon=2$ when $5\mid b$. Suppose $5\mid b$ then the conductor of $\rhoJpp$ at $\Fq_5$ does not degenerate upon reduction mod $\Fp$ by applying \cite[Theorem 1.5]{Jarvis}. For the case $5\mid a$ we get that the Serre level of $\brhoJpp$ divides $\Fq_5$.

By Theorem \ref{conductor-minus} for $(\rho_{J^-,\lambda} \otimes \chi_0)$ we have that
\begin{equation*}
  N(\rho_{J^-,\lambda} \otimes \chi_0) = \Fq_2^s \cdot \Fq_5^\epsilon \cdot \prod_{\substack{\Fq \mid ab, \\ \Fq \not= \Fq_2, \Fq_5}} \Fq,
\end{equation*}  
where $\epsilon=2,3$ and $s = 1, 2$. Therefore, the Serre level of $\brhoJmptwo$ divides $\Fq_2^s \cdot \Fq_5^\epsilon$. 
The conductor of $\rhoJmptwo$ at $\Fq_5$ does not degenerate upon reduction mod $\Fp$ by applying \cite[Theorem 1.5]{Jarvis}, hence the conductor of $\brhoJmptwo$ at $\Fq_5$ is still $\Fq_5^\epsilon$. Finally, by Theorem \ref{conductor-minus} the conductor of $\brhoJmptwo$ at $\Fq_2$ divides $\Fq_2$ when $2\mid a$ and $\Fq_2^2$ when $2\mid c$.

\section{Irreducibility of $\brhoJpmp$}

In this section, we provide general local results which are used to prove irreducibility of the representations $\rhobar_{J^\pm,\Fp}$ with efficient bounds on $p$. We refer the reader to \cite{Dim, BCDF1} for the general strategies for proving irreducibility using local methods with asymptotic bounds on $p$.

We start with $K$ being a finite extension of $\Q_p$, having ring of integers $\calO$, uniformizer $\pi$, and residue field $k$. Fix embeddings $\Q_p \subseteq K \subseteq \overline{\Q}_p$ and $\F_p \subseteq k \subseteq \overline{\F}_p$. For any integer $n\geq1$, denote by $\F_{p^n}$ the subfield of~$\overline{\F}_p$ with~$p^n$ elements. Let $I_K$ denote the inertia subgroup of $G_K$ and $P_K \subseteq I_K$ the wild inertia subgroup. Let $I_{t,K} = I_K/P_K$ denote the tame inertia group of $K$. We denote by $G_{L/K}$ the Galois group of a Galois extension $L/K$. Let $I_{L/K}$ and $I_{t,L/K}$ the inertia and tame inertia subgroups of $G_{L/K}$.

The action of $I_{t,K}$ on $\pi^{1/(p^n-1)}$ gives a homomorphism $\psi : I_{t,K} \rightarrow \F_{p^n}^\times \subseteq \overline{\F}_p^\times$, denoted $\theta_{p^n-1}$ as in~\cite{serre}, which we refer to as \textit{the fundamental character of level $n$}. In contrast, \textit{a fundamental character of level $n$} is any conjugate over $\F_p$ of $\psi = \theta_{p^n-1}$, that is, a character of the form $\sigma \circ \psi$, where $\sigma : \F_{p^n} \hookrightarrow \overline{\F}_p$ is an embedding. In particular, there are $n$ distinct fundamental characters of level $n$.

\begin{theorem}
We have
\begin{equation*}
  I_{t,K} \cong \varprojlim_{(d,p)=1} \mu_d \cong \varprojlim_m \F_{p^m}^\times,
\end{equation*}
where the projective limits are over the homomorphisms
\begin{align*}
\mu_{dd'} & \longrightarrow \mu_d \\
  \alpha & \longmapsto \alpha^{d'}
\end{align*}
and
\begin{align*}
\F_{p^{m n}}^\times & \longrightarrow \F_{p^m}^\times \\
  \alpha & \longmapsto \alpha^{\frac{1-p^{m n}}{1-p^m}}.
\end{align*}
\end{theorem}

\begin{proof}
  See \cite[Section~1.2]{serre}.
\end{proof}

\begin{lemma}
\label{invert}

Suppose $k \cong \F_{p^n}$. Let $\psi = \theta_{p^n-1}$ be the fundamental character of level $n$. 

\begin{itemize}
\item Suppose $L$ is a finite tamely ramified abelian extension of $K$ such that $\psi$ factors as $\psi = \psi_{L/K} \circ \alpha$ where
\begin{equation*}
  \psi_{L/K} : I_{t,L/K} \cong I_K \cdot G_{L}/G_{L} \cong I_K/(I_K \cap G_{L}) \rightarrow \F_{p^n}^\times
\end{equation*}
and
\begin{equation*}
  \alpha : I_{t,K} \rightarrow I_{t,L/K}
\end{equation*}
is the natural homomorphism. 
\item Let $r_{L/K} : K^\times \rightarrow G_{L/K}$ be the local reciprocity map, whose restriction to $\calO^\times$ factors through a map ${\bar r}_{L/K} : k^\times \rightarrow G_{L/K}$.
\end{itemize}

Then we have that
\begin{equation*}
  \psi_{L/K} \circ {\bar r}_{L/K} (\bar x) = {\bar x}^{-1}
\end{equation*}
for all $x \in \F_{p^n}^\times$.
\end{lemma}
\begin{proof}
From~\cite[Prop.~3]{serre}, we have that
\begin{equation*}
  {\bar r}_{L/K} \circ \psi(s^{-1}) = \alpha(s)
\end{equation*}
for all $s \in I_{t,K}$. Hence, we have
\begin{equation*}
  \psi_{L/K} \circ {\bar r}_{L/K} \circ \psi(s^{-1})  = \psi_{L/K} \circ \alpha(s) = \psi(s)
\end{equation*}
for all $s \in I_{t,K}$. Putting $x = \psi(s^{-1})$ and noting that $\psi$ is surjective to $\F_{p^n}^\times$ yields the result.
\end{proof}

\begin{corollary}
\label{torus-cor}
Suppose $\varphi : G_K \rightarrow \overline{\F}_p^\times$ is a continuous homomorphism. Then, we have that
\begin{equation*}
  \varphi \circ r_K(x) = \prod_{\sigma} \sigma(\bar x)^{n(\sigma)},
\end{equation*}
where $\bar x \in k^\times$ is the reduction of $x \in \calO^\times$, $r_K\colon K^{\times}\rightarrow G_K^{ab}$ is the local reciprocity map, $0 \le n(\sigma) \le p-1$, and $\sigma$ runs through the embeddings of $k \hookrightarrow \overline{\F}_p$.
\end{corollary}
\begin{proof}
Let $p^n = |k|$. Firstly, we can write
\begin{equation}
\label{torus-map}
  \varphi \mid_{I_K} = \prod_{\sigma} (\sigma \circ \psi)^{-n(\sigma)}
\end{equation}
where $\psi = \theta_{p^n-1}$ and $0 \le n(\sigma) \le p-1$ as
\begin{equation}
  \varphi \mid_{I_K} = \psi^k
\end{equation}
for some $0 \le -k \le p^n-2$ as $\psi$ has order $p^n-1$. Write $k$ in the form
\begin{equation}
    k = a_0 + a_1 p + \ldots + a_{n-1} p^{n-1}
\end{equation}
where $0 \le -a_i \le p-1$ for all $0 \le i \le n-1$, then the conjugates $\sigma \circ \psi$ are of the form $\psi^{p^i}$ where $0 \le i \le n-1$.

Taking $L = K(\zeta_{p^n-1})(\pi^{1/(p^n-1)}) = K(\pi^{1/(p^n-1)})$, which is a finite tamely ramified abelian extension over $K$, we see from Lemma~\ref{invert} that the fundamental character $\psi$ of level $n$ has the description
\begin{equation}
\label{embed-label}
  \psi \circ r_K(x) = \bar{x}^{-1}.
\end{equation}

The desired result then follows from precomposing \eqref{torus-map} with $r_K$.

\end{proof}

Now let $K$ be a number field. The following result is a generalization of \cite[Appendice A]{kraus8}. We provide additional details for the benefit of the reader. 
\begin{proposition}\label{L:KrausAppendix}
Let $K$ be a number field with ring of integers~$\calO_K$. Let~$p$ be a prime number unramified in~$K$ and~$S_p$ the set of places in~$K$ above~$p$. Let $\varphi : G_K \to \overline{\F}_p^\times$ be a continuous homomorphism satisfying the following conditions:
\begin{enumerate}
\item The Artin conductor of $\varphi$ is $\m$, an ideal of $\calO_K$ prime to $p$;
\item For all~$\Fp \in S_p$, the restriction $\varphi \mid_{I_{K_\Fp}}$ is equal to $\displaystyle\prod_{\sigma \in \Omega_\Fp} (\sigma \circ \psi_\Fp)^{-n_\Fp(\sigma)}$, where 
 \begin{enumerate} 
   \item $0 \le n_\Fp(\sigma) \le p-1$,
   \item $k_\Fp$ is the residue field of $\Fp$,
   \item $\Omega_\Fp$ denotes the set of embeddings of $k_\Fp$ into $\overline{\F}_p$,
   \item $\psi_\Fp : I_{t,K_\Fp} \rightarrow k_\Fp^\times \subseteq \overline{\F}_p^\times$ is the fundamental character, where $K_\Fp$ is the completion of $K$ at $\Fp$.
 \end{enumerate}
\end{enumerate}
Then, for all totally positive units~$u \in \calO_K^\times$ such that $u \equiv 1 \pmod{\m}$, we have that
\begin{equation*}
  \prod_{\Fp \in S_p} \prod_{\sigma \in \Omega_\Fp} \sigma (u + \Fp)^{n_\Fp(\sigma)} = 1.
\end{equation*}
\end{proposition}
\begin{proof}
Let $r_K : \Aff_K^\times \rightarrow G_K^\text{ab}$ be the global reciprocity map, and let $K_p = (K \otimes \Q_p)^\times = \prod_{\Fp \in S_p} K_\Fp^\times$, which sits inside the id\`ele group  $\Aff_K^\times$. We also denote by 
$$
U_\m=\lbrace x\in\Aff_K^\times:~x_v\in\calO_v^\times,~x_v>0\text{ for all real }v \text{ and }x_\Fq\equiv 1\pmod{\Fq^{v_{\Fq}(\m)}}\text{ for all }\Fq\mid \m\rbrace.
$$

We have that
\begin{enumerate}
    \item $\varphi \circ r_K$ is trivial on $U_{\m,v}$ for places $v \nmid p$,
    \item $\varphi \circ r_K(x) = \prod_{\Fp \in S_p} \prod_{\sigma \in \Omega_\Fp} \sigma (x_\Fp + \Fp)^{n_\Fp(\sigma)}$ for $x = \prod_\Fp x_\Fp \in K_p^\times$ by Corollary~\ref{torus-cor}.
\end{enumerate}

It follows that $\varphi \circ r_K$ is trivial on $E_\m = U_\m \cap K^\times$, that is, the group of totally positive units $u \in \calO_K^\times$ such that $u \equiv 1 \pmod \m$.
\end{proof}

\begin{corollary}
\label{unit-condition}
    With the notation and hypotheses of Proposition~\ref{L:KrausAppendix} in place, if $n_\Fp(\sigma) = 0$ except for one pair $(\Fp_0, \sigma_0)$, then we have that
    \begin{equation}
        \sigma_0(u)^{n_{\Fp_0}(\sigma_0)} \equiv 1 \pmod{\Fp_0},
    \end{equation}
    for every totally positive unit $u \in \calO_K^\times$ such that $u \equiv 1 \pmod{\m}$.
\end{corollary}

We now have all the ingredients to prove $\brhoJpp$ and $\brhoJmp$ are irreducible. 

\begin{theorem}
\label{irreducible-plus} Suppose $2 \nmid ab$ and $5 \mid ab$. Then, $\brhoJpp$ is irreducible for $p > 5$.
\end{theorem}

\begin{proof}
Since $\brhoJpp$ is odd and $K$ is totally real it is well known that $\brhoJpp$ is absolutely irreducible if and only if it is irreducible.

From Lemma~\ref{trivial-char}, we have that $\det \brhoJpp = \chi_p$. Suppose $\brhoJpp$ is reducible, that is,
\begin{equation*}
\brhoJpp \simeq \begin{pmatrix} \theta & \star \\ 0 & \theta' \end{pmatrix} 
\quad \text{with} \quad \theta, \theta' : G_K \rightarrow \F_{\Fp}^*
\quad \text{satisfying} \quad \theta \theta' = \chi_p,
\end{equation*}
where $\F_{\Fp}$ is the residual field of $K$ at $\Fp$. As $\theta\theta'=\chi_p$, the characters $\theta$ and $\theta'$ have the same conductor exponents away from $p$.

Let $\Fq\neq\Fp$ be a prime of $\OK$. The semi-simplification $\brhoJppsemi$ of a reduction $\brhoJpp$ does not depend on the choice of lattice and its restriction to $I_{K_{\Fq}}$ is isomorphic to $(\theta \oplus \theta') \mid_{I_{K_{\Fq}}}$. Suppose $e_{\Fq}$ is the conductor exponent of $\theta$ and $\theta'$ which is the same as mentioned above. Then, the conductor exponent at $\Fq$ of $\brhoJppsemi$ is $2e_{\Fq}$. On the other hand, the conductor exponent at $\Fq$ of $\brhoJppsemi$ is bounded by the conductor exponent at $\Fq$ of $\brhoJpp$ which is $0$ for $\Fq\neq\Fq_5$ and $0, 1, 2$ for $\Fq=\Fq_5$ by Theorem \ref{thm:Serre_level}. Thus, $e_{\Fq}=0$ for all $\Fq\neq\Fq_5$ and the conductor of $\theta$ and $\theta'$ away from $p$ divides $\Fq_5$.

By Theorem \ref{thm:Serre_level} again, the representation $\brhoJpp$ is finite at all primes $\Fp \mid p$, and as $p$ is unramified in $K$, it follows from \cite[Corollaire 3.4.4]{Raynaud} that the restriction to $I_{K_\Fp}$ of $\theta \oplus \theta^\prime$ is isomorphic to (up to switching $\theta$ and $\theta'$)
\begin{equation}\label{eq:Theta_plus_semi}
 1 \oplus \chi_p \qquad\text{or}\qquad \psi_\Fp \oplus \psi_\Fp^p
\end{equation}
where $\chi_p$ is the $p$th cyclotomic character and $\psi_\Fp$ is the fundamental character of level $2$ of $I_{t,K_\Fp}$.

Suppose $\theta$ is unramified at all primes $\Fp \mid p$. Then, $\theta$ corresponds to a character of the Ray class group of modulus $\Fq_5\infty_1\infty_2$ where $\infty_i$ denote the two places at infinity, which is isomorphic to $\Z/2\Z$. Hence, the Ray class group modulus $\Fq_5\infty_1\infty_2$ corresponds to a quadratic character $\psi.$ It follows that $\theta=\psi^t$ and $\theta^\prime = \psi^t\chi_p$ where $t=0,1$. It holds $\psi(\Frob_{\Fq_2})=\pm 1$ and $\chi_p(\Frob_{\Fq_2})=2^2$. As $C^+/K$ has good reduction at $\Fq_2$, we have that
\begin{equation*}
  \pm(1 + 2^2) \equiv a_{\Fq_2}(J^+)\pmod {\Fp},
\end{equation*}
where $a_{\Fq_2}(J^+)$ is the trace of $\rhoJpp$ evaluated at a Frobenius element at $\Fq_2$. Using the change of variables in \cite[Proposition 1.15]{DarmonDuke} with a short \Magma~script,
we check that $a_{\Fq_2}(J^+)=0$, so $\Fp\mid 5$ which is a contradiction to the fact that $p>5$.

Suppose that both $\theta$ and $\theta^\prime$ ramify at some prime above $p$. Applying Corollary~\ref{unit-condition} over $K$ with $\varphi$ equal to either $\theta$ or $\theta^\prime$ implies that $p$ divides the norm of $u - 1$ by \eqref{eq:Theta_plus_semi}, where $u = \epsilon_1$ is the fundamental unit in $K$. However, this norm is $-1$, a contradiction.
\end{proof}

\begin{theorem}\label{irreducible-minus}
Suppose $2 \mid a$, $b \equiv -1 \pmod 4$ or $2 \mid c$, $a \equiv -1 \pmod 4$, and $5 \nmid ab$. Then, $\brhoJmp$ is irreducible for $p >5$.
\end{theorem}

\begin{proof}
Since $\brhoJmp$ is odd and $K$ is totally real, $\brhoJmp$ is absolutely irreducible if and only if it is irreducible. To show $\brhoJmp$ is irreducible, it suffices to show $\brhoJp$ is irreducible where $J = J^- \otimes \chi_0$ and $\chi_0$ is as in Theorem~\ref{conductor-minus}.

From Lemma~\ref{trivial-char}, we have that $\det\brhoJp = \chi_p$. Suppose $\rhobar_{J,\Fp}$ is reducible, that is,
\begin{equation*} 
\rhobar_{J,\Fp} \simeq \begin{pmatrix} \theta & \star \\ 0 & \theta' \end{pmatrix} 
\quad \text{with} \quad \theta, \theta' : G_K \rightarrow \F_{\Fp}^*
\quad \text{satisfying} \quad \theta \theta' = \chi_p,
\end{equation*}
where $\F_{\Fp}$ is the residual field of $K$ at $\Fp$. As $\theta\theta'=\chi_p$, the characters $\theta$ and $\theta'$ have the same conductor exponents away from $p$. 

Let $\Fq\neq \Fp$, $\Fq_2$ and $\Fq_5$. From Theorem \ref{thm:Serre_level} we know that $\rhobar_{J,\Fp}$ is unramified at $\Fq$. As in Theorem \ref{irreducible-plus} we conclude that $e_{\Fq}=0$ where $e_{\Fq}$ is the conductor exponent of $\theta$ and $\theta^\prime$ at $\Fq$.




The conductor at $\Fq_2$ of the semi-simplification of $\rhobar_{J,\Fp}$, which is isomorphic to $\theta \oplus \theta'$, divides the conductor at $\Fq_2$ of the semi-simplification of $\rho_{J,\Fp}$, which in turn divides the conductor at $\Fq_2$ of $\rho_{J,\Fp}$, namely, $\Fq_2^s$ where $s=1,2$. As $\theta$ and $\theta'$ have the same conductor at $\Fq_2$, it follows that the conductors at $\Fq_2$ of $\theta$ and $\theta'$ divide $\Fq_2$.

From Proposition \ref{semistable_at_5_J_minus} we know that $\rhoJmp(I_{K_{\Fq_5}})$ has order equal to $4$ or $20$, therefore  $\rhobar_{J,\Fp}(I_{K_{\Fq_5}})$ has order equal to $4$ or $20$, respectively. This holds because $\rho_{J,\Fp}(I_{K_{\Fq_5}})$ does not intersect the kernel of reduction (which is a pro-$p$ group) since $p > 5$ and $\chi_0$ is unramified at $\Fq_5$. 

Since $\brhoJp (I_{K_{\Fq_5}})$ has order coprime to $p > 5$ by Proposition~\ref{semistable_at_5_J_minus} and Maschke's Theorem, we conclude that $\rhobar_{J,\Fp} \mid_{I_{K_{\Fq_5}}} \simeq \theta \oplus \theta' \mid_{I_{K_{\Fq_5}}}$. From Theorem \ref{thm:Serre_level}, we know that the conductor exponent at $\Fq_5$ of $\brhoJp\mid_{I_{K_{\Fq_5}}}$ is $2,3$. As above we conclude that the conductors at $\Fq_5$ of $\theta$ and $\theta'$ are the same and divide $\Fq_5$. 

In summary, the conductor of $\theta$ and $\theta'$ away from $p$ divides $\Fq_2\Fq_5$. By Theorem \ref{thm:Serre_level}, the representation $\brhoJmp$ is finite at all primes $\Fp \mid p$, and as $p$ is unramified in $K$, it follows from \cite[Corollaire 3.4.4]{Raynaud} that the restriction to $I_{K_\Fp}$ of $\theta \oplus \theta^\prime$ is isomorphic to (up to switching $\theta$ and $\theta'$)
\begin{equation}\label{eq:Theta_plus_semi_minus}
 1 \oplus \chi_p  \qquad \text{or} \qquad \psi_\Fp \oplus \psi_\Fp^p,
\end{equation}
where $\chi_p$ is the $p$-th cyclotomic character and $\psi_\Fp$ is the fundamental character of level $2$ of $I_{t,K_\Fp}$.

Suppose $\theta$ is unramified at all primes $\Fp \mid p$. Then, $\theta$ corresponds to a character of the Ray class group of modulus $\Fq_2\Fq_5\infty_1\infty_2$ where $\infty_i$ denote the two places of $K$ at infinity. The Ray class group with this modulus is isomorphic to $\Z/2\Z$. From the above we have that $4\mid\#\rhobar_{J^-,\Fp}(I_{K_{\Fq_5}})$ and $\rhobar_{J^-,\Fp} \mid_{I_{K_{\Fq_5}}} \cong \theta \oplus \theta' \mid_{I_{K_{\Fq_5}}}$, hence $\theta\mid_{I_{K_{\Fq_5}}}$ has order divisible by $4$, which is a contradiction.

Suppose that both $\theta$ and $\theta^\prime$ ramify at some prime above $p$. Let $\epsilon_1$ be
the fundamental unit in $K$. Applying Corollary~\ref{unit-condition} over $K$ with $\varphi$ equal to either $\theta$ or $\theta^\prime$ implies that $p$ divides the norm 
of $u - 1$ by \eqref{eq:Theta_plus_semi_minus}, where $u = \epsilon_1^{n_1}$ and $n_1 = 12$ is the smallest positive integer
such that $u \equiv 1 \pmod{\Fq_2\Fq_5}$. This yields of bound of $p = 2,5$, contradicting $p > 5$.
\end{proof}

\section{Proof of Theorem~\ref{main-theorem}}

Let $(a,b,c) \in \Z^3$ be a non-trivial primitive solution to \eqref{main-equ}. It is enough to prove Theorem \ref{main-theorem} for the case $n=p$ an odd prime or $4$.  For the case $n=3$ this is a result by Kraus \cite{Kraus1}, the case $n=5$ is a special case of Fermat's Last Theorem (see for instance, \cite[Th\'eor\`eme IX]{Di28}) while the case $n=4$ has been proved in \cite[Theorem 1.1]{Bruin03}. The case $n = 7$ is treated in \cite{dahmen-siksek}.

For the rest of the proof, we may now assume that $n=p\geq 11$ is a prime. Let $S_2(\mathfrak{n})$ denote the space of Hilbert newforms over $K$ with parallel weight $2$, trivial character, and level $\mathfrak{n}$. 

\subsection{Local comparison of traces}

In the elimination step of the modular method using a Frey abelian variety $J$, we typically show that an isomorphism
\begin{equation}
\label{rhobar-isom}
    \rhobar_{J,\Fp} \simeq \rhobar_{g,\mathfrak{P}},
\end{equation}
where $\Fp$ is prime of $K$ and $\mathfrak{P}$ is a prime of field of coefficient $K_g$ of a Hilbert newform $g$, cannot occur by exhibiting a prime $\Fq$ not dividing $\Fp$ and the Artin conductor of $\rhobar_{J,\Fp}$ such that
\begin{equation}
\label{trace-compare-naive}
    \tr \rhobar_{J,\Fp}(\Frob_\Fq) \not= \tr \rhobar_{g, \mathfrak{P}}(\Frob_\Fq).
\end{equation}
However, a subtlety occurs because in the definition of the isomorphism \eqref{rhobar-isom} we mean
\begin{equation}
    \rhobar_{J,\Fp} \otimes \overline{\F}_p \simeq \rhobar_{g,\mathfrak{P}} \otimes \overline{\F}_p.
\end{equation}
Hence, the comparison \eqref{trace-compare-naive} cannot be done until we have chosen an embedding of the residue fields of $K_\Fp$ and $K_{g,\mathfrak{P}}$ into $\overline{\F}_p$. To rule out an isomorphism as in \eqref{rhobar-isom} by a local comparison of traces it is necessary to average out this ambiguity and instead use the condition
\begin{equation}
    p \nmid N_{L/\Q}\left( \prod_{\sigma \in \Gal(K/\Q)} (a_\Fq(g) - a_\Fq(J)^\sigma) \right),
\end{equation}
where $L$ is the compositum of $K$ and $K_g$ inside $\overline{\Q}$ (see \cite{BillereyChenDieulefaitFreitasNajman} for more details). Here we define the quantities
\begin{align}
  a_\Fq(J) & = \tr \rho_{J,\Fp}(\Frob_\Fq), \\
  a_\Fq(g) & = \tr \rho_{g, \mathfrak{P}}(\Frob_\Fq).
\end{align}
Finally, we remark that this does not affect the computational time for the elimination step, since in practice {\tt Magma} is only able to compute $a_\Fq(J)$ up to Galois conjugation over $\Q$ in any case.

\subsection{Proof of Theorem~\ref{main-theorem} (I)}

We are under the assumption that $2 \nmid ab$ and $5 \mid ab$. By Theorems \ref{irreducible-plus}, \ref{modularity-plus}, and Lemma~\ref{trivial-char}, we have that $\brhoJpp$ is irreducible and modular with trivial character. Hence, by level lowering for Hilbert modular forms~\cite{Fuj,Jarv,Raj}, we have that 
\begin{equation}\label{elimination-plus}
  \brhoJpp\simeq \rhobar_{g,\mathfrak{B}},
\end{equation}
where $g$ is a Hilbert newform of parallel weight $2$, trivial character over $K$ and level $\calO_K$, $\Fq_5$ or $\Fq_5^2$ by Theorem \ref{thm:Serre_level}, and $\mathfrak{B}$ is a prime above $p$ in the field of coefficients of $g$. However, the spaces of Hilbert newforms $S_2(1)$, $S_2(\Fq_5)$ and $S_2(\Fq_5^2)$ are empty which gives a contradiction.

\subsection{Proof of Theorem~\ref{main-theorem} (II)}

As before, we may assume that $n=p\geq 11$ is a prime. We may assume without loss of generality that $2 \mid a$, $b \equiv -1 \pmod 4$ by switching the roles of $a$ and $b$ and negating $(a,b,c)$, if necessary.

By Theorems \ref{irreducible-minus}, \ref{modularity-minus}, and Lemma~\ref{trivial-char}, we have that $\brhoJmp$ is irreducible and modular with trivial character.

In what follows, we use the fact that
\begin{equation*}
    \rho_{J^- \otimes \chi, \Fp} = \rho_{J^-, \Fp} \otimes \chi,
\end{equation*}
for any character $\chi : G_K \rightarrow \left\{ \pm 1 \right\}$ of order dividing $2$, where $J^- \otimes \chi$ means the twist of $J^-$ by $\chi$, and $\chi$ twists by the automorphism $-1$ of $J^-$.

Let $J = J^- \otimes \chi_0$ where $\chi_0\in K(S_2, 2)^*$ and $S_2=\lbrace \Fq_2\rbrace$. By level lowering for Hilbert modular forms as in the proof of Theorem~\ref{main-theorem} (I), we have that
\begin{equation}\label{eq:elimination-minus}
  {\bar \rho}_{J,\Fp} \simeq \rhobar_{g,\mathfrak{B}},
\end{equation}
where $g$ is a Hilbert newform of parallel weight $2$, trivial character over $K$ and level $\Fq_2 \Fq_5^\epsilon$ where $\epsilon = 2, 3$, by Theorem \ref{thm:Serre_level}, and $\mathfrak{B}$ is a prime of $K$ above $p$ in the field of coefficients of $g$. Note in level lowering, we can level lower prime by prime and choose not to strip $\Fq_2$ from the level of $\rho_{J,\Fp}$ in case the Serre level of ${\bar \rho}_{J,\Fp}$ is not divisible by $\Fq_2$.

Suppose $\q\neq \p$ is a prime of $K$ not above $2$ and $5$. Then we have that
\begin{align}
a_\q(g) & \equiv a_{\q}(J) \pmod{\p}, &~\text{if }q\nmid ab,\\
a_\q(g)^2 & \equiv (N(\q) + 1)^2 \pmod{\p}, &~\text{if } q\mid ab, 
\end{align}
where $a_{\q}(J)=\text{tr}{\rho}_J(\Frob_\q)$ and $N(\q)$ is the norm of $\q$. Thus, defining
$$
T(g,\q)= N(\q)\cdot \left(a_\q(g)^2-(N(\q) + 1)^2 \right)\cdot\prod_{a,b\in\mathbb{F}_q,ab\neq 0} N_{L/\Q} \left( \prod_{\sigma \in \Gal(K/\Q)} (a_\q(g) - a_{\q}(J)^\sigma )\right),
$$
we have that $p\mid T(g,\q)$, where $L$ is the compositum of $K$ and $K_g$ inside $\Qbar$. Taking the $\gcd$ of $T(g,\q)$ for a suitable choice of primes $\q$ above $q$, we typically obtain a small finite set of possible primes $p$ (assuming one of the $T(g,\q)$ is non-zero). 

For all choices of $\chi_0\in K(S_2, 2)^*$, using the auxiliary primes $q = 3, 7, 11$, we eliminate all newforms $g$ except for the prime exponents $p = 2, 3, 5, 7$, which proves the desired conclusion.

\begin{remark}
Working with all the $J^-\otimes \chi_0$ instead of the $J^-$ does have two important consequences. First of all, the level of the space of Hilbert newforms we have to compute is smaller, i.e.\ otherwise we would have to compute the space of Hilbert newforms of level $\Fq_2^4\Fq_5^\epsilon$ with $\epsilon=2,3$ for instance. Secondly, the spaces of Hilbert newforms of level $\Fq_2^4 \Fq_5^\epsilon$ with $\epsilon=2,3$ contain additional forms with complex multiplication over $\Q(\zeta_5)$. The elimination step using a standard comparison of traces of Frobenius does not work on these forms and requires additional elimination techniques to prove Theorem~\ref{main-theorem}.
\end{remark}

\section{Proof of Theorem~\ref{thm:eliminate-to-CM}}
As before, we may assume that $n=p\geq 11$ is a prime. 

The proof is similar to the proof of Theorem~\ref{main-theorem} (II) so we outline the differences according to the cases involved. Again, let $J=J^-\otimes\chi_{0}$.

If $2 \mid ab$ and $5 \mid ab$ and $\rhobar_{J,\Fp}$ is irreducible, then $\rhobar_{J,\Fp}$ arises from level $\Fq_2 \Fq_5^2$. There are two newforms at level $\Fq_2 \Fq_5^2$. We eliminate both newforms as in the case (II) above. So we conclude that $\rhobar_{J,\Fp}$ is reducible, and hence $\rhobar_{J^-,\Fp}$ is reducible.

If $2 \nmid ab$ and $5 \nmid ab$, then we may assume without loss of generality that $2 \mid c$, $a \equiv -1 \pmod 4$ by switching the roles of $a$ and $b$ and negating $(a,b,c)$, if necessary. We now have the additional levels $\Fq_2^2 \Fq_5^\epsilon$ for $\epsilon = 2, 3$ to consider. All forms except one at level $\Fq_2^2 \Fq_5^2$; and three at level $\Fq_2^2 \Fq_5^3$, can be eliminated for prime exponents $p$ not equal to $2,3,5,19$. These forms are obtained from $J^-(-8,8,0) \otimes \chi_2$, where $\chi_2$ corresponds to $K(\sqrt{2})/K$, and $J^-(a,b,c)$ for $(a,b,c) = (-1,1,0), (-4,4,0), (-16,16,0)$, respectively. This can be verified by computing the conductors for these specific Jacobians using \cite{DokchitserDoris19} and noting they are modular at the correct levels.

\bibliography{main}{}
\bibliographystyle{plain}
\end{document}